\documentclass[reqno]{amsart}
\usepackage{amssymb,amsmath}
\usepackage[all]{xy}
\usepackage{graphicx}

\numberwithin{equation}{section}

\newtheorem{theorem}{Theorem}[section]
\newtheorem{definition}[theorem]{Definition}

\newtheorem{lemma}[theorem]{Lemma}
\newtheorem{remark}[theorem]{Remark}
\newtheorem{prop}[theorem]{Proposition}

\newcommand{\AI}{A_\infty}

\newcommand{\OT}{\otimes}

\newcommand{\CC}{\mathbb{C}}
\newcommand{\RR}{\mathbb{R}}
\newcommand{\QQ}{\mathbb{Q}}
\newcommand{\kk}{\boldsymbol{k}}
\newcommand{\ZZ}{\mathbb{Z}}
\newcommand{\E}{\epsilon}

\newcommand{\CO}{\mathcal{O}}

\newcommand{\CL}{\mathcal{L}}

\newcommand{\WH}[1]{\widehat{#1}}
\newcommand{\HH}[1]{\widehat{#1}}
\newcommand{\WT}[1]{\widetilde{#1}}
\newcommand{\OL}[1]{\overline{#1}}
\newcommand{\UL}[1]{\underline{#1}}
\newcommand{\wtd}[1]{\widetilde{#1}}
\newcommand{\NOV}{\Lambda_{nov}}
\newcommand{\NOVO}{\Lambda_{0,nov}}
\newcommand{\PD}[1]{\frac{\partial}{\partial x^{#1}}}

\begin{document}

\title{Notes on Kontsevich-Soibelman's theorem\\about cyclic $\AI$-algebras}
\author{Cheol-hyun Cho and Sangwook Lee}
\address{Cheol-Hyun Cho, Department of Mathematics and Research Institute of Mathematics, Seoul National University,
Kwanakgu, Seoul, South Korea, Email: chocheol@snu.ac.kr, 
Sangwook Lee, Department of Mathematics, Seoul National University, Email: leemky7@snu.ac.kr, }
\begin{abstract}
Kontsevich and Soibelman has proved a relation between a non-degenerate cyclic homology element of an $\AI$-algebra $A$ and its cyclic inner products on the minimal model of $A$. We find an explicit formula of this correspondence, in terms of the strong homotopy inner products and negative cyclic cohomology of $A$. We prove that an equivalence class of the induced strong homotopy inner product depends only on the given negative cyclic cohomology class. Also, we extend such a correspondence to the case of gapped filtered $\AI$-algebras.
\end{abstract}
\thanks{Authors were supported by the Basic research fund 2009-0074356 funded by the Korean government}
\maketitle

\section{Introduction}
$\AI$-algebra is a strong homotopy associative algebra introduced by Stasheff \cite{St}.
Cyclic symmetry of an $\AI$-structure (also called Calabi-Yau property) plays a fundamental role in its homological algebra. It also
plays an important role in homological mirror symmetry and open topological field theory.
It was first introduced by Kontsevich \cite{K}, as an invariant non-commutative symplectic structure of the dual formal manifold. 
In this correspondence,  tensor coalgebra corresponds to
the formal non-commutative manifolds, $\AI$-algebra structure (codifferential) to the formal vector field $Q$ with $[Q,Q]=0$.
This dual point of view turns out to be rather powerful and Kontsevich and Soibelman has proved the following remarkable theorem.

\begin{theorem}[\cite{KS} Theorem 10.2.2]\label{thm:KS}
For weakly unital, compact $\AI$-algebra $A$, cyclic cohomology class which is homologically
non-degenerate gives rise to a class of isomorphisms of cyclic inner products on a minimal model of $A$.
\end{theorem}

In this paper, we find an explicit correspondence of the above relation in Theorem \ref{thm}. For the explicit relations, we use the strong homotopy inner products introduced by the first author in \cite{C} which are homotopy notion of cyclic inner products. This relates the notion of cyclic inner products with the $\AI$-bimodule maps satisfying certain additional properties(see Definition \ref{def:cyc2}).
Also we find that it is better for the explicit relations to work with negative cyclic cohomology, instead
of cyclic cohomology (cf. Lemma \ref{hchc}). In particular one can directly work with Hochschild cohomology classes. This correspondence should be useful for the explicit computations of correlation functions of open TCFT's 
such as Landau-Ginzburg models or Fukaya categories.

We prove that for a homologically non-degenerate negative cyclic
cohomology class (in particular a Hochschild cohomology class) of an $\AI$-algebra $A$, we find a corresponding homotopy cyclic inner product explicitly on $A$ (not just the minimal model of $A$). 

\begin{theorem}\label{thm}
  For a weakly unital compact $\AI$-algebra $A$, a homologically nondegenerate negative cyclic cohomology class $[\phi]$ gives rise to an isomorphism class of strong homotopy inner products on $A$. In particular, from $[\phi]$, we construct an $\AI$-bimodule map $ \WT{\phi_0}:A \to A^*$ explicitly using the Proposition \ref{prop:first} which gives a strong homotopy inner product.

In particular, we have an $\AI$-algebra $B$ with the cyclic inner product $\psi:B \to B^*$, with an quasi-isomorphism $\iota:B\to A$, satisfying 
 the commuting diagram \begin{equation}\label{defdiagram2} 
\xymatrix{A \ar[d]_{\WT{\phi_0}} & B \ar[l]_{\WT{\iota}} \ar[d]^{cyc}_{\psi} \\ A^* \ar[r]^{\WT{\iota}^*} & B^* }
\end{equation}
\end{theorem}
In proving the theorem, we found that the lemma \ref{lemma:vQ} (due to \cite{KS}) may need more detailed explanation
which is carried out in the Lemma \ref{lem:homfromv}, \ref{homfromv2} and Proposition \ref{prop:welldefined} in this paper. This involves a rather subtle point, and
actually is a crucial step to prove the classification results.  For example, when we apply the Darboux theorem, by making the symplectic form to be constant symplectic form, the $\AI$-structure is also changed in the process. To prove that isomorphic (negative) cyclic cohomology class give rise to isomorphic inner products, $\AI$-structure needs to be fixed while transforming the symplectic structure. We explain how to obtain the actual
automorphism for such a procedure.

In \cite{KS}, it is also mentioned that $\psi_{0,0}$ can be given by the trace map, and we explain its
relation to the above theorem in section \ref{sec:compare}.

Our original motivation was to have an analogous theorem in the case of filtered $\AI$-algebras so that
we can apply such theorems to the $\AI$-algebra of Lagrangian submanifolds
where cyclicity plays an important role (see for example \cite{Fu3},\cite{Fu4}). 
This has been carried out in section \ref{sec:filter}, where we consider the case of gapped filtered $\AI$-algebras with
Novikov ring coefficients which has energy filtrations. 
It turns out one should be careful to deal with filtered notions for example in the Darboux theorem \ref{darboux}.
But we find that the main theorem and many of the explicit correspondences extend to the case of filtered $\AI$-algebras if we consider differential forms with non-negative energy coefficients. We consider the formal manifold language in the filtered case using the topological dualization of the bar complex as in \cite{C1}. Until the last section, we will not consider filtered case, and hence we have $m_0=0$.

{\bf Acknowlegements}
This is a result of our effort to understand the work of Kontsevich and Soibelman as shown in the title.

\section{$\AI$-algebras and formal manifolds}
We recall the basic notions of $\AI$-algebras and its dualization, namely formal manifolds.
\subsection{$\AI$-algebra}
We recall the definition of an $\AI$-algebra(\cite{St}).
Let $\kk$ be the field containing $\QQ$ (for example $\QQ,\RR,\CC$) with $char(\kk)=0$.
Let $C=\bigoplus_{j\in \mathbb{Z}} C^{j}$ be a graded
vector space over $\kk$, and consider its suspension $( C[1])^m =  C^{m+1}$
and $|x_i|'$ is the shifted grading $|x_i| -1$.
The \textbf{tensor-coalgebra}
of $C[1]$ over $\kk$ is given by
$ BC:=\bigoplus_{k\geq 1} T_k(C[1]),$
where \begin{equation}\label{def:tk}
T_k(C[1]) = \underbrace{C[1] \otimes \cdots \otimes C[1]}_{k},
\end{equation}
with the comultiplication $\Delta:BC\longrightarrow BC\otimes BC$ defined by
\begin{equation}\label{comul}
 \Delta(v_{1}\otimes \cdots \otimes v_{n}):=\sum_{i=1}^{n} (v_{1}\otimes \cdots \otimes v_{i})
    \otimes(v_{i+1}\otimes \cdots \otimes v_{n}).
\end{equation}
Now, consider a family of maps of degree one
$$m_k : T_k(C[1]) \to C[1], \ \ \textrm{for} \ k =1,2,\cdots.$$
We can extend $m_k$ uniquely to a coderivation
\begin{equation}\label{eq:hatd}
\WH{m}_k(x_1 \otimes \cdots \otimes x_n) = \sum_{i=1}^{n-k+1}
(-1)^{|x_1|' + \cdots + |x_{i-1}|'} x_1 \otimes \cdots \otimes m_k(x_i,
\cdots, x_{i+k-1}) \otimes \cdots \otimes x_n
\end{equation}
for $k \leq n$ and $\WH{m}_k(x_1 \otimes \cdots \otimes x_n) =0$ for $k >n$.

The coderivation $\WH{d} = \sum_{k=1}^\infty \WH{m}_k$ is well-defined as a map
from $BC$ to $BC$.
The $\AI$-equations are equivalent to the equality
$\WH{d} \circ \WH{d} =0$, or equivalently,
\begin{definition}\label{def:ai}
 An $\AI$-algebra $(C,\{m_*\})$ consists of a $\ZZ$-graded vector space $C$ over $\kk$ with
a collection of multi-linear maps  $m:= \{ m_n:C[1]^{\otimes n} \to C[1]\}_{n\geq 1}$ of
degree one satisfying the following equation for each $k=1,2,\cdots.$
\begin{equation}\label{aiformula}
0=\sum_{k_1+k_2=k+1} \sum_{i=1}^{k_1-1} (-1)^{\epsilon_1} m_{k_1}(x_1,\cdots,x_{i-1},m_{k_2}(x_i,\cdots,x_{i+k_2-1}),\cdots,x_k)
\end{equation}
where $\epsilon_1= |x_1|'+ \cdots + |x_{i-1}|'$.
\end{definition}
An element $I \in C^0 = C^{-1}[1]$ is called a unit if
\begin{equation}\label{unit}
\begin{cases}
 m_{k+1}(x_1,\cdots,I,\cdots,x_k) = 0 \;\; \textrm{for} \; k \neq 2 \\
m_2(I,x) = (-1)^{deg \, x} m_2(x,I) = x.
\end{cases}
\end{equation}

We recall the notion of an $\AI$-homomorphism between two $\AI$-algebras.
The family of maps of degree 0
$$f_k : B_k(C_1) \to C_2[1] \;\;\textrm{for} \; k =1,2,\cdots $$
induce the coalgebra map
$\HH{f}:BC_1 \to BC_2$, which for $x_1 \otimes  \cdots \otimes x_k \in B_k C_1$ is defined
by the formula
$$\HH{f}(x_1\otimes \cdots \otimes x_k) = \sum_{1 \leq k_1 \leq \cdots \leq k_n \leq k}
f_{k_1}(x_1,\cdots,x_{k_1})\otimes \cdots \otimes f_{k-k_n }(x_{k_n+1},\cdots,x_{k}).$$
The map $\HH{f}$ is called an $\AI$-homomorphism if  
$$\HH{d} \circ \HH{f} = \HH{f} \circ \HH{d}.$$

An $\AI$-algebra $(C,\{m_k\})$ is called {\em compact} if $H^\bullet(C,m_1)$ is finite dimensional
and is called {\em minimal} if $m_1 \equiv 0$.
\subsection{Formal manifolds}
\label{sec:noncomm}
We recall the language of formal non-commutative geometry mainly from \cite{KS}, which provides more geometric point of view of the related homotopy theories. For a more systematic exposition, we refer readers to Kontsevich and
Soibelman(\cite{KS}), Kajiura \cite{Kaj} or  Hamilton and Lazarev \cite{HL}.

We restrict ourselves to $\AI$-algebras on a finite dimensional vector space $V$ over a field $\kk$.
(To prove the main theorem for compact $\AI$-algebras, we will pull-back all the related notions
to $H^\bullet(C,m_1)$ which is finite dimensional).
We choose a basis of $V$ and denote it as $\{e_1,\cdots,e_n\}$. Consider the dual space $V^*=Hom(V,\kk)$ and denote the dual basis as $x_1,\cdots,x_n$ whose degrees are  given as $|x_i|'=-|e_i|'$.

As a $\AI$-algebra is given by coalgebra with codifferential, its dual becomes (non-commutative) differential graded algebra (DGA) or a formal manifold in the language of Kontsevich and Soibelman \cite{KS}. Namely, the dual of the coalgebra,  $(BC)^*$ is DGA. To have a unit,
actually, one should work with an augmented bar complex 
$BC^+ = BC \oplus \kk$ with the comultiplication $\Delta:BC^+ \to BC^+ \otimes BC^+$
defined as in (\ref{comul}) with the sum from $i=0$ to $i=n$. 
Then, the dual space  $(BC^+)^*$ may be considered as a function ring $\mathcal{O}(X)$ such that $$\CO (X) = \kk \ll x_1,...,x_n \gg.$$
Namely, it is a \emph{noncommutative} formal power series ring, regarded as the ring of regular functions on the formal non-commutative manifold $X$ 

Dual of the codifferential $\HH{d}$ becomes a differential of the DGA, which may be considered as a vector field on $X$.
The vector field $Q$ on $X$ corresponding $\HH{d}$ may be defined by
$$Q := \sum_{k >0} a_i^{j_1,\cdots,j_{k}} x_{j_{1}} x_{j_{2}} \cdots x_{j_k} \frac{\partial}{\partial x_i}$$
where the coefficients $a$'s are defined  by the $\AI$-operations
$$m(e_{j_1},\cdots,e_{j_k}) = \sum a_i^{j_1,\cdots,j_k} e_i.$$
Note that $Q$ may be an infinite sum and is regarded as a formal vector field.

The $\AI$-equation $\HH{d} \circ \HH{d} = \frac{1}{2}[\HH{d}, \HH{d}] = 0$ implies the relation $[Q,Q]=0$ or the following identities between the coefficients of $Q$ for each $s$
$$0= \sum_{\stackrel{k_1+k_2 =k+1}{j,l}} (-1)^{\E_1}a_s^{i_1,\cdots,i_{j-1},l,i_{j+k_2},\cdots,i_k} \cdot a_l^{i_j,\cdots,i_{j+k_2-1}}$$

Here $\PD{i}$ acts on $\mathcal{O}(X)$ in a natural way:
for example, (assume each of $k(x),f(x)$ and $g(x)$ has homogeneous degree)
$$k(x) \PD{1} (f(x)g(x)) =
 k(x)\PD{1}(f(x)) \cdot g(x) + (-1)^{(|g(x)|')|f(x)|'} k(x)\PD{1}(g(x)) \cdot f(x).$$
If we set $|\frac{\partial}{\partial x_i}|' = -|x_i|'$, then one may check that $\AI$-equation corresponds to $[Q,Q]= 0$
or more precisely, for any $f(x)$, we have
$$[Q,Q](f) = Q(Q(f)) -(-1)^{(1\cdot 1)} Q(Q(f)) = 2Q(Q(f))=0.$$

Hence, the structure of an $\AI$-algebra on $V$ is equivalent to the non-commutative (pointed) formal manifold $X$
equipped with a vector field $Q$ with $[Q,Q]=0$. Here pointed means that
we consider the case of $m_0=0$ or in another way, we consider formal series with no constant term.

A cohomomorphism between $\AI$-algebras naturally corresponds to the algebra homomorphism compatible with derivations. 
Namely for two $\AI$-algebras $(A, m^A_*),(B,m^B_*)$ and 
a $\AI$-homomorphism $h:B \to A$, the formal change of coordinate of the dual variables are given as follows.
We assume $B$ is finite dimensional as a vector space, and denote by $\{f_*\}$ its basis, and 
introduce corresponding formal variables $y_*$ as before.
Suppose 
$$h_{k}(f_{j_1},\cdots,f_{j_k}) = h^i_{j_1,\cdots,j_k} e_i, \;\;\; h^i_{j_1,\cdots,j_k} \in R.$$
Then, algebra homomorphism is defined by changing each variable as 
\begin{equation}\label{changeco}
	x_i \mapsto h^i_{j_{11}} y_{j_{11}} + h^i_{j_{21},j_{22}} y_{j_{21}} y_{j_{22}} + \cdots +
	h^{i}_{j_{l1},\cdots,j_{lk}} y_{j_{l1}}\cdots y_{j_{lk}} + \cdots.
\end{equation}
We refer readers to \cite{Kaj} for
detailed explanation on this point.

\subsection{Non-commutative De Rham theory}
There is the non-commutative version of de Rham (or Karoubi) theory (see for example \cite{KS}).
The main difference from the commutative case is that
 the space $X$ where the differential forms should live, does not really exist, and
 the right de Rham complex in the non-commutative case is the cyclic de Rham complex.
 
First, one may introduce the de Rham forms as follows.
Consider $\CO(T[1]X):=\kk \ll x_i,dx_i \gg$, where $dx_i$ are another formal variables such that $|dx_i|=|x_i|$.
There are additional signs when dealing with these forms or vector fields, which we follow the definition in \cite{Kaj}.
First denote $$\sharp(dx_i)=1, \sharp(x_i)=0, \sharp (\PD{i}) = -1 $$
and in general, by denoting $ x_i$ or $dx_i$ by $\phi$,
one defines $$ \sharp (\phi^1 \cdots \phi^k) = \sum_{j=1}^k \sharp(\phi^j).$$
And the Koszul rule in this case is given by considering the sign $\sharp$ and $|\cdot|'$ separately.
For example, graded commutator is defined (for homogeneous elements) by
$$[f(\phi),g(\phi)] = f(\phi)g(\phi) - (-1)^{\big(|f(\phi)|'|g(\phi)|' + \sharp(f(\phi))\sharp(g(\phi))\big)} g(\phi)f(\phi).$$

Cyclic functions are defined by the quotient
$$\Omega^0_{cyc}(X)= \CO (X) / [\CO (X),\CO (X)]_{top}.$$
Here one takes the closure of algebraic commutator in the adic topology.
The space of \textit{cyclic differential forms} on $X$ is defined similarly by
$$\Omega_{cyc}(X)= \CO (T[1]X) / [\CO (T[1]X),\CO (T[1]X)]_{top}.$$
Cyclic non-commutative one forms on $X$
, $\Omega^1_{cyc}(X)$ is  generated by expressions as $x_{i_1} \cdots x_{i_k}dx_{i_{k+1}}$,
where by cyclic rotation, $dx_*$ may be regarded as being in the last slot. But in general,
cyclic 2-form is generated by equivalence classes of elements like
$$x_{i_1}\cdots x_{i_p} dx_{a} x_{j_1} \cdots x_{j_q} dx_{b} x_{k_1}\cdots x_{k_r}.$$
Hence it is easy to see that $\Omega^s_{cyc}(X)$ does not vanish for $s > dim(V)$.
The usual de Rham differential $d$ descends to the quotient $\Omega_{cyc}(X)$ and
we denote it as $d_{cycl}$ as in \cite{KS}. $(\Omega_{cyc}(X),d_{cycl})$ is the
non-commutative de Rham complex.

Furthermore, it is well-known that the \textit{contraction map}(or \textit{interior product}) $i_{\xi}:\CO (T[1]X) \to \CO (T[1]X)$ can be defined by $i_{\xi}(f)=0$, $i_{\xi} (df) = \xi(f)$ for all $f \in \CO (T[1]X)$.
Now, one defines the Lie derivative
$$\CL _{\xi} = [d,i_{\xi}] = d \circ i_{\xi} + i_{\xi} \circ d.$$
As $\CL_{\xi}$ is also a derivation, we have for any $f(\phi),g(\phi) \in   \CO (T[1]X)$,
$$\CL_{\xi}([f(\phi),g(\phi)]) \subset [\CO (T[1]X),\CO (T[1]X)].$$
Hence, $\CL_{\xi}$ is well-defined on cyclic forms $\Omega_{cyc}(T[1]X)$.
Like the standard differential calculus, the following holds true also for the non-commutative de Rham complex.
$$ [d,d]=0, [d,\CL _{\xi}]=0,$$
$$[\CL _{\xi}, i_{\eta}] = i_{[\xi,\eta]}, [\CL_{\xi},\CL_{\eta}]=\CL_{[\xi,\eta]},[i_{\xi},i_{\eta}]=0.$$

We mention two of the well-known theorems. The first one is 
\begin{theorem}[Poincar\'e lemma] The cohomology of $(\Omega_{cyc}(X),d)$ is trivial.
\end{theorem}
\begin{proof}
We follow Lemma 4.8 of \cite{Kaj}. One can define the explicit contracting homotopy $H$ satisfying $dH + Hd= Id$ as follows.
Denote an element of $\Omega_{cyc}(X)$ as $a = \frac{1}{k}a_{i_1 \cdots i_k}\phi_{i_1}\cdots \phi_{i_k}$
where $\phi_{i_j}=x_{i_j}$ or $\phi_{i_j}=dx_{i_j}$.
Then, $H$ is defined by
$$H(a) = \sum_{i}(-1)^{\sharp_{i_1} + \cdots \sharp_{i_{j-1}}} \frac{1}{k}a_{i_1 \cdots i_k}\phi_{i_1}\cdots (H (\phi_{i_j})) \cdot \phi_{i_k}. $$
where $H(x_{i_j})=0$ and $H(dx_{i_j}) =x_{i_j}$.
\end{proof}
\begin{theorem}(Darboux theorem)
Any symplectic form on a formal noncommutative manifold can be transformed to the constant (coefficient) symplectic form by a coordinate transformation.
\end{theorem}
We refer readers to \cite{G} or \cite{Kaj} or the Theorem \ref{darboux} for its proof in the filtered case.

\subsection{Hochschild (co)homology for $\AI$-algebras}

Hochschild homology of an $\AI$-algebra $A=(C,m)$ as an $\AI$-module over itself is defined as follows. Denote
\begin{equation}\label{def:hh}
C^k(A,A) = C \otimes C[1]^{\otimes k},
\end{equation}
and its degree $\bullet$ part by $C^k_\bullet(A,A)$. We define the Hochschild chain complex
\begin{equation}\label{def:hochchain}
(C_{\bullet}(A,A),b) = \big( \oplus_{k \geq 0} C^k_\bullet(A,A),b \big),
\end{equation}
where the degree one differential $b$ is defined as follows: for $v \in A$ and $x_i \in A$,
$$b (\UL{v} \otimes x_1 \otimes \cdots \otimes x_k)
= \sum_{\stackrel{0 \leq j \leq k+1 -i}{1 \leq i}} (-1)^{\E_1} \UL{v} \otimes \cdots \otimes x_{i-1} \otimes m_j(x_i,\cdots,x_{i+j-1}) \otimes \cdots \otimes x_k$$
\begin{equation}\label{eq2}
 + \sum_{\stackrel{0 \leq i, j \leq k}{ i+j \leq k}} (-1)^{\E_2} \UL{m_{i+j+1}\big(x_{k-i+1},\cdots,x_k,v, x_1,\cdots,x_{j} \big)} \otimes x_{j+1} \otimes \cdots
\otimes x_{k-i}
\end{equation}
We underlined the module element and the signs follow the Koszul sign convention: $$\E_1 = |v|' + |x_1|' + \cdots + |x_{i-1}|',\;
\E_2 = \big(\sum_{s=1}^i |x_{k-i+s}|'\big)\big( |v|' + \sum_{t=1}^j |x_{t}|'\big).$$

Similarly, one can define the reduced Hochschild homology by considering $C^{red}:=C/R \cdot I$ and
set $C^k_{red}(A,A) = C \otimes (C^{red}[1])^{\otimes k}$ instead, and the resulting homology is isomorphic to the standard one.

The cochain complex obtained by taking a dual of the reduced Hochschild chain complex, $\big((C_{\bullet}^{red}(A,A))^*,b^*\big)$ defines the Hochschild cohomology $H^\bullet_{red}(A,A^*)$. Here, cochain elements are given by the maps $\{f_n:(C^{red}[1])^{\OT n} \to C^*\}$ and the degree one differential $b^*$ is given by
$$b^*f(a_1,...,a_n)=\sum (-1)^{Kos} f(a_1,...,m_k(...),...,a_n)+\sum (-1)^{Kos} d^*(a_1,...,\UL{f(...)},...,a_n),$$
Here $A^*$ has a canonical $\AI$-bimodule structure over $A$ whose bimodule structure is given by $d^*$ in \ref{def3}.

\subsection{(Negative) cyclic cohomology for $\AI$-algebras}
We recall the definition of the cyclic and negative cyclic cohomology of a unital $\AI$-algebra.
(For weakly unital case, see for example, \cite{HL} or \cite{C1})
In the case of $\AI$-algebras, there exist an Tsygan's bicomplex, and also $(b,B)$-complex defining cyclic homology. Cyclic and negative cyclic cochain complex can be obtained by taking the
dual of the $(b,B)$-complex in the following way.

Consider the following unbounded $(b^*,B^*)$ complex.
 \begin{equation} \label{diag:periodic}
  \xymatrix{
  \hspace{1cm} & \hspace{1cm} & \hspace{1cm} & \hspace{1cm} \\
    \; \ar[r]^{B^*} & C^{2}_{red}(A,A^*) \ar[r]^{B^*} \ar[u]^{b^*} & C^{1}_{red}(A,A^*) \ar[r]^{B^*} \ar[u]^{b^*} & C^{0}_{red}(A,A^*) \ar[r]^{B^*} \ar[u]^{b^*} & \; \\
  \; \ar[r]^{B^*} & C^{1}_{red}(A,A^*) \ar[r]^{B^*} \ar[u]^{b^*} & C^{0}_{red}(A,A^*) \ar[r]^{B^*} \ar[u]^{b^*} & C^{-1}_{red}(A,A^*) \ar[r]^{B^*} \ar[u]^{b^*} & \; \\
  \; \ar[r]^{B^*} & C^{0}_{red}(A,A^*) \ar[r]^{B^*} \ar[u]^{b^*} & C^{-1}_{red}(A,A^*) \ar[r]^{B^*} \ar[u]^{b^*} & C^{-2}_{red}(A,A^*) \ar[r]^{B^*} \ar[u]^{b^*} & \; \\
  \hspace{1cm} & \hspace{1cm} \ar[u]^{b^*}  &  \hspace{1cm} \ar[u]^{b^*}  &  \hspace{1cm} \ar[u]^{b^*} \\
  }
\end{equation}
Here $(C^{\bullet}_{red}(A,A^*),b^*)$ is the dual of reduced Hochschild cochain complex in the previous subsection, and
the dual of the Conne-Tsygan's operator is given by
$$B^*f(a_1,...,a_n)=\sum_{i}f(a_i,a_{i+1},...,a_{i-1})(I).$$

The double complex $CC^{\bullet}_{-}(A,A^*)$ which consists of only negative columns of (\ref{diag:periodic}) defines the 
\textit{negative cyclic cohomology} of an $\AI$-algebra $A$, which we denote by $HC^\bullet_{-}(A)$ . Here, elements in the line parallel to $y=-x$ line has the same total degree of the double complex, and we consider the direct sums instead of the direct products so that
its dualization would give the negative cyclic homology which is given by the direct products (with suitable finiteness assumption).

The double complex $CC^{\bullet}(A,A^*)$ which consists of only nonnegative columns of (\ref{diag:periodic}) defines the cyclic cohomology of the $\AI$-algebra $A$, which we denote by $HC^\bullet(A)$. Here we use the direct products instead of direct sums so that the dualization of the
cyclic homology would give cyclic cohomology.

The double complex $CP^{\bullet}(A,A^*)$ which consists of all columns of (\ref{diag:periodic}) defines the periodic cyclic cohomology of the $\AI$-algebra $A$, which we denote by $HP^\bullet(A)$. Here we use the direct products.

As usual, one obtains the following spectral sequence of these three homology theories given by the inclusion of
$CC^{\bullet}(A,A^*)$ to $CP^{\bullet}(A,A^*)$:
\begin{equation}\label{eq:les}
\cdots \to HC^n(A) \to HP^n(A) \to HC^n_{-}(A) \to HC^{n+1}(A) \to \cdots 
\end{equation}
Here the map $HC^n_{-}(A) \to HC^{n+1}(A)$ is induced by $B^*$

These homology theories for arbitrary $\AI$-algebra can be difficult to deal with, and we are mainly interested in the case that the $\AI$-algebra
$A=(C,\{m_k\})$ satisfies either $C^{> 0} \equiv 0$ or $C^{<0} \equiv 0$ before shifting degrees.
We remark that the usual (non-graded) algebras may be considered as $\AI$-algebras and  after degree shifting, all elements have degree $-1$.
In this case the Hochschild complex $C_{\geq 0}(A,A) \equiv 0$ for degree reasons. The examples from geometry, for example the usual de Rham complex, has degree from $0$ to $N$.  In particular, if we assume that $C^0$ is generated by the unit (in cohomology), then it is easy to that 
the Hochschild cochain complex satisfies $C^{>1}_{red}(A,A^*) \equiv 0$. We also remark that by the standard spectral sequence arguments,
homotopy equivalent $\AI$-algebras has isomorphic Hochschild (co)homology classes. 
As we have used direct sums to define negative cyclic cohomology, the usual 
usual invariant, coinvariant relation gives rise to the following lemma:
\begin{lemma}[\cite{H} Lemma 3.6]\label{hchc}
 Let $(A,m)$ be a weakly unital $\AI$-algebra for which there exists an integer $N$ such that $H^k(V,V^*) =0$ for $k>N$. Then, for any integer $n$, we have 
$$HC^n_{-}(A) \cong HC^{n+1}(A).$$
\end{lemma}
The equivalence is given by the map in the long exact sequence (\ref{eq:les}), and this is also the relation between
the cohomology classes used in Theorem \ref{thm:KS} and Theorem \ref{thm}.

\section{Cyclic $\AI$-algebras and its homotopy notions}
The cyclic structure was first considered by Kontsevich \cite{K} as a symplectic form on the non-commutative formal manifolds (see section \ref{sec:noncomm}).
There are somewhat different sign conventions whether one is working with degree shifting or not and
we refer readers to \cite{C} for detailed explanation on this point.
\subsection{Cyclic $\AI$-algebras}
\begin{definition}
An $\AI$-algebra $(C,\{m_*\})$ is said to have a {\it cyclic symmetric} inner product if
there exists a skew-symmetric non-degenerate, bilinear map $$<,> : C[1] \otimes C[1] \to R,$$
such that for all integer $k \geq 1$,
\begin{equation}\label{cyeqn}
    <m_{k,\beta}(x_1,\cdots,x_k),x_{k+1}> = (-1)^{K}<m_{k,\beta}(x_2,\cdots,x_{k+1}),x_{1}>.
\end{equation}
 where $K = |x_1|'(|x_2|' + \cdots +|x_{k+1}|')$.
 For short, we will call such an algebra, cyclic $\AI$-algebra.
\end{definition}
There is a notion of cyclic $\AI$-homomorphism due to Kajiura \cite{Kaj}
\begin{definition}
An $\AI$-homomorphism $\{h_k\}_{k\geq 1 }$ between two cyclic $\AI$-algebras is called
a cyclic $\AI$-homomorphism if
\begin{enumerate}
\item $h_1$ preserves inner product $<a,b> = <h_1(a),h_1(b)>$.
\item \begin{equation}
\sum_{i+j=k} <h_i(x_1,\cdots,x_i), h_j(x_{i+1},\cdots,x_k)> =0.
\end{equation}
\end{enumerate}
\end{definition}

\subsection{Strong homotopy inner products}
We recall the notion of a strong homotopy inner product of $\AI$-algebra $A$ from \cite{C} in this section.
(see also \cite{CL}).

For $\AI$-algebras, cyclicity is not preserved under quasi-isomorphism of an $\AI$-algebra and the homotopy cyclicity
has been defined in \cite{C}. Its characterization theorem and relation to non-commutative symplectic forms has been explored.
For such notions, it is useful to use $\AI$-bimodule maps from $A$ to its dual $A^*$ which was
called as infinity inner product in \cite{T}.

We first recall a canonical $\AI$-bimodule structure on the dual $A^*$ of an $\AI$-algebra $(A,m)$.
We refer readers to \cite{FOOO} for the definition of $\AI$-bimodule.
Recall that $\AI$-bimodule structure of $M$ over $A$ is described by sequence of maps of degree one
$$d_{k,l}:A[1]^{\otimes k} \otimes M[1] \otimes A[1]^{\otimes l} \to M[1].$$
For the case of $M=A[1]$, we may set $d_{k,l}= m_{k+l+1}$.
For the case of the dual $M = (A[1])^*$, we define the bimodule structure $d^*_{k,l}$ as follows
\begin{equation}\label{def3}
d^*_{k,l}(x_1,\cdots,x_k,v^*,x_{k+1},\cdots,x_{k+l}) (w) =
(-1)^{\epsilon} v^* \big(  m_{k+l+1} (x_{k+1},\cdots,x_{k+l},w,x_1,\cdots,x_k) \big),
\end{equation}
with $\epsilon = Kos = |v^*|' + (|x_1|'+ \cdots + |x_k|')(|v^*|' + |x_{k+1}| + \cdots + |x_{k+1}|' + |w|')$.

We recall the following lemma relating cyclic inner product and $\AI$-bimodule maps.
\begin{lemma}[\cite{C} Lemma 3.1]
Let $\psi$ be an $\AI$-bimodule homomorphism $\psi:A \to A^*$.
Define $$<a,b>=\psi_{0,0}(a)(b),$$
and suppose that $<,>$ is {\em non-degenerate}.
Then, it defines a cyclic symmetric inner product on $A$  if
\begin{enumerate}\label{cycprop}
\item $ \psi_{k,l} \equiv 0$ for $(k,l) \neq (0,0)$
\item $\psi_{0,0}(a)(b) = -(-1)^{|a|'|b|'}\psi_{0,0}(b)(a).$
\end{enumerate}
Conversely, any cyclic symmetric inner product $<,>$ on $A$ give rise to an $\AI$-bimodule map
$\psi:A \to A^*$ with (1) and (2).
\end{lemma}

We also recall that $\AI$-homomorphism $f:A\to B$ can be also understood as an $\AI$-bimodule homomorphism $\widetilde{f}:A\to B$ over $(f,f)$.
In this case, $\widetilde{f}_{k,l}:A^{\otimes k} \otimes A \otimes A^{\otimes l} \to B $ is defined by $\widetilde{f}_{k,l} = f_{k+l+1}$.
One can check that $\widetilde{f}$ satisfies $ \widetilde{f} \circ \widehat{m}^A = m^B \circ \WH{\widetilde{f}}$.

Now, we recall the definition of strong homotopy inner product from \cite{CL} which is modified from \cite{C},
\begin{definition}\label{def:cyc2}
    Let $A$ be an $A_{\infty}$-algebra. We call an $A_{\infty}$-bimodule map $\phi:A \rightarrow A^*$ a
    {\textbf{strong homotopy inner product}} if  it satisfies the following properties.
    \begin{enumerate}
    \item (Skew symmetry) $\phi_{k,l}(\vec{a},\UL{v},\vec{b})(w)=- (-1)^{Kos}\phi_{l,k}(\vec{b},\UL{w},\vec{a})(v)$.
    \item (Closedness) for any choice of a family $(a_1,...,a_{l+1})$ and any choice of indices $1 \leq i<j<k \leq l+1$, we have $$\phi(...,\UL{a_i},...)(a_j)+(-1)^{Kos}\phi(...,\UL{a_j},...)(a_k)+(-1)^{Kos}\phi(...,\UL{a_k},...)(a_i)=0.$$
    \item (Homological non-degeneracy) for any non-zero $[a] \in H^\bullet(A)$ with $a \in A$, there exists a $[b] \in H^\bullet(A)$ with $b \in A$,  such that $\phi_{0,0}(a)(b) \neq 0$.
\end{enumerate}
And $A$ is called {\em homotopy cyclic} $\AI$-algebra, if there exists a strong homotopy inner product of $A$.
\end{definition}

Then, the main result of \cite{C} can be phrased as  the following theorem.
\begin{theorem}\label{prop:shi}
 Let $\phi:A \rightarrow A^*$  be an $A_{\infty}$-bimodule map.
\begin{enumerate}
\item If $\phi$ is a strong homotopy inner product,
then there exists an $\AI$-algebra $B$ with a cyclic inner product
    $\psi: B \rightarrow B^*$
        and an $\AI$-quasi-isomorphism $\iota:B \rightarrow A$ satisfying the
    following commutative diagram of $A_{\infty}$-bimodule homomorphisms
\begin{equation}\label{defdiagram1}
\xymatrix{A \ar[d]_{\phi} & B \ar[l]_{\tilde{\iota}} \ar[d]^{cyc}_{\psi} \\ A^* \ar[r]^{{\tilde{\iota}}^*} & B^* }
\end{equation}

\item If there exists a cyclic $\AI$-algebra $B$ with $\psi:B \to B^*$ and
an $\AI$-quasi-isomorphism $f:A \to B$ such that the following diagram of $\AI$-bimodules over $A$ commutes
\begin{equation}\label{defdiagram3}
\xymatrix@C+1cm{ A \ar[r]_{ g = \widetilde{f}}  \ar[d]_\phi & \, B \ar[d]_\psi^{cyc} \\ A^*  &  \ar[l]^{g^*} B^*}	 
\end{equation}
then, $\phi$ is a strong homotopy inner product.
\end{enumerate}
If $\phi_{0,0}$ is non-degenerate in the chain level,  then one can find $B$ such that both diagrams
(\ref{defdiagram1}), (\ref{defdiagram3}) holds.
\end{theorem}

We also recall the equivalence of two strong homotopy inner product from \cite{CL}.
\begin{definition}\label{equivst}
Two strong homotopy inner products $\phi:A \to A^*$ and $\psi:B \to B^*$ are said to be \textit{equivalent} if
there exists a cyclic minimal $\AI$-algebra $H$ with a quasi-isomorphism to $A$ and $B$, with the following
commutative diagram:
     $$\xymatrix{A \ar[d]_{\phi} & H \ar[d]^{cyc} \ar[l]_{qis} \ar[r]^{qis} & B \ar[d]^{\psi} \\
            A^* \ar [r] & H^* & B^* \ar[l]}$$
\end{definition}

We remark that general $\AI$-homomorphism do not preserve cyclic property of $\AI$-algebra.(but strong homotopy
inner product is given via the diagram \ref{defdiagram3}.) The notion of a cyclic $\AI$-homomorphism, which
preserves cyclic property of $\AI$-algebra, was first considered by Kajiura \cite{Kaj} from the condition
the $f^\omega=\omega'$ so that both $\omega$ and $\omega'$ are constant coefficient symplectic forms.

\section{Correspondences between algebra and formal noncommutative geometry}
It is useful to develop a "dictionary" between notions in homological algebras and that of formal manifolds.
First, the fact that cyclic symmetry of an $\AI$-algebra can be understood as certain symplecic forms is well-known and
originally due to Kontsevich.
\begin{lemma}
Cyclic $\AI$-algebra $(V,\phi)$, namely $\phi$ is a cyclic inner product on $V$, is equivalent to
a non-commutative constant symplectic two form $\omega$
with $\CL_Q \omega =0$.
\end{lemma}
\begin{proof}
  Let $\displaystyle \omega=\sum_{a,b} \omega_{ab}(dx^a dx^b)_c$, where $\phi(e_a,e_b)=\omega_{ab}$. Then
  \begin{eqnarray}
    & & \CL_Q \omega = \CL_Q (\sum_{a,b} \omega_{ab}(dx^a dx^b)_c) \nonumber \\ \nonumber
    &=& \sum_{a,b}((\omega_{ab} (\CL_Q dx^a) dx^b)_c + (dx^a (-1)^{|a|'|Q|'}\CL_Q dx^b)_c) \\\nonumber
    &=& \sum_{a,b} (\omega_{ab}\sum_{i_1,...,i_k}\sum_{1 \leq l \leq k}m^a_{i_1\cdots i_k}(x^{i_1}\cdots dx^{i_l} \cdots x^{i_k} dx^b)_c) \\\nonumber
    & & +(-1)^{|a|'}\sum_{i_1,...,i_k}\sum_{1 \leq l \leq k} m^b_{i_1 \cdots i_k}(dx^a x^{i_1} \cdots dx^{i_l} \cdots x^{i_k})_c \\\nonumber
    &=& \sum_{a,b} \omega_{ab}\sum_{i_1,...,i_k}\sum_{1 \leq l \leq k} m^a_{i_1\cdots i_k}(x^{i_1}\cdots dx^{i_l} \cdots x^{i_k} dx^b)_c \\\nonumber
    & & +\sum_{a,b}(-1)^{|b|'} \omega_{ba} (-1)^{1+|b|'(|i_1|'+\cdots + |i_k|')}\sum_{i_1,...,i_k}\sum_{1 \leq l \leq k}(x^{i_1}\cdots dx^{i_l} \cdots x^{i_k} dx^b)_c \\\nonumber
    &=& \sum_{a,b}\sum_{i_1,...,i_k}\sum_{1 \leq l \leq k} \omega_{ab}(1+(-1)^{1+|b|'(1+|i_1|'+\cdots + |i_k|')+|a|'|b|'+1})m^a_{i_1\cdots i_k}(x^{i_1}\cdots dx^{i_l} \cdots x^{i_k} dx^b)_c \\
    &=& \sum_{a,b}\sum_{i_1,...,i_k}\sum_{1 \leq l \leq k} 2 \omega_{ab}m^a_{i_1\cdots i_k}(x^{i_1}\cdots dx^{i_l} \cdots x^{i_k} dx^b)_c. \label{eq:2omega}
  \end{eqnarray}
  Note that $(-1)^{1+|b|'(1+|i_1|'+\cdots + |i_k|')+|a|'|b|'+1}=1$ because  $|a|'=|i_1|'+\cdots + |i_k|'+1$ by the fact that $Q$ has degree 1. A careful observation on the cyclic monomials in (\ref{eq:2omega}) leads us to the following: $\CL_Q \omega=0$ is equivalent to
  \begin{eqnarray*}
    & & \omega_{ab}m^a_{i_1\cdots i_k}(x^{i_1}\cdots dx^{i_l} \cdots x^{i_k} dx^b)_c \\
    & & + \omega_{a i_l} m^a_{i_{l+1} \cdots i_k b i_1 \cdots i_{l-1}}(x^{i_{l+1}} \cdots x^{i_k} dx^b x^{i_1} \cdots x^{i_{l-1}} dx^{i_l})_c \\
    &=& (\omega_{ab} m^a_{i_1\cdots i_k} + (-1)^p \omega_{a i_l} m^a_{i_{l+1} \cdots i_k b i_1 \cdots i_{l-1}})(x^{i_1} \cdots dx^{i_l} \cdots x^{i_k} dx^b)_c \\
    &=& 0
  \end{eqnarray*}
  for $p=1+(|i_1|'+\cdots +|i_l|')(|i_{l+1}|'+\cdots +|i_k|'+|b|')$, if and only if
  $$\omega_{ab} m^a_{i_1\cdots i_k} = (-1)^{p+1} \omega_{a i_l} m^a_{i_{l+1} \cdots i_k b i_1 \cdots i_{l-1}},$$
  i.e.
  \begin{eqnarray*}
    & & <m(e_{i_1},\cdots,e_{i_k}),e_b> \\
    &=& (-1)^{(|i_1|'+\cdots +|i_l|')(|i_{l+1}|'+\cdots +|i_k|'+|b|')} <m(e_{i_{l+1}},\cdots,e_{i_k},e_b,e_{i_1},\cdots,e_{i_{l-1}}),e_l>,
  \end{eqnarray*}
  which is the cyclicity.
\end{proof}

Recall that we have $ \CL_Q \circ \CL_Q = 0$ and hence, on de Rham complex $\Omega_{cyc}(X)$, we have two differentials $d_{cyc}$ and $\CL_Q$.
By the Poincar\'e lemma, the homology with respect to $d_{cyc}$ is trivial, and interesting homologies are given by the differential $\CL_Q$.

\begin{lemma}\label{lem:compare}
For a unital finite dimensional $\AI$-algebra $A$,
$(\Omega^1_{cyc}(X)[1], \CL_Q)$ can be identified with Hochschild cochain complex $(C^\bullet(A,A^*),b^*)$,
and $(\Omega^0_{cyc}(X)/\kk, \CL_Q)$ can be identified with cyclic cochain complex $((C^\lambda(A))^*,b^*)$.

Namely, we have the following 1-1 correspondences.\\

\begin{tabular}{|c|c|}\hline
    \hspace{1.5cm} $\AI$-algebra $A$ \hspace{1.5cm} & Formal noncommutative manifold $X$ \\\hline \hline
    $\eta \in C^\bullet(A,A^*)$ & $\alpha_\eta \in \Omega^1_{cyc}(X)$\\\hline
    $b^* \eta$ & $\CL_Q \alpha_\eta$ \\\hline
    $\xi \in (C^\lambda_\bullet(A))^*$ & $f_\xi \in \Omega^0_{cyc}(X)$\\\hline
    $b^* \xi$ & $\CL_Q f_\xi$ \\\hline
\end{tabular}
\end{lemma}
\begin{proof}
We first check the statement for Hochschild cochains. The degree shifting $[1]$ is the result due to
the choice of the chain complex in \ref{def:hh}. If $\eta \in \mathrm{Hom}(A[1]^{\otimes n},A^*)$ given by $\eta(e_{i_1},...,e_{i_n})(e_j)=\eta_{i_1,...,i_n}^{j}$ for basis elements $e_*$, 
it corresponds to the 1-form $\alpha_{\eta}=\sum \eta_{i_1,...,i_n}^{j}(x^{i_1}\cdots x^{i_n} dx^j)_c$.
We omit the Koszul signs in the following formulas.
We verify that $b^* \eta$ corresponds to $\mathcal{L}_Q \alpha_{\eta}$. 
\begin{eqnarray*}
  b^*\eta (e_{i_1},...,e_{i_n})(e_j)&=& \sum \eta(e_{i_1},...,m_k(e_{i_l},...,e_{i_{l+k-1}}),...,e_{i_n})(e_j) \\
  & & + \sum \eta(e_{i_l},...,e_{i_{l+p}})(m_k(e_{i_{l+p+1}},...,e_{i_n},e_j,e_{i_1},...,e_{i_{l-1}})) \\
  &=&  \sum_{q} \eta_{i_1,...,i_{l-1},q,i_{l+k},...,i_n}^j \cdot m_{i_l,...,i_{l+k-1}}^q \\
  & & +  \sum_q \eta_{i_l,...,i_{l+p}}^q \cdot m_{i_{l+p+1},...,i_n,j,i_1,...,i_{l-1}}^q.
\end{eqnarray*}

Thus,
\begin{eqnarray*}
  \alpha_{b^* \eta} &=& \sum (\sum_q \eta_{i_1,...,i_{l-1},q,i_{l+k},...,i_n}^j \cdot m_{i_l,...,i_{l+k-1}}^q \\
  & & + \sum_q \eta_{i_l,...,i_{l+p}}^q \cdot m_{i_{l+p+1},...,i_n,j,i_1,...,i_{l-1}}^q)x^{i_n} \cdots x^{i_1} dx^j
\end{eqnarray*}
is the 1-form corresponding to $b^* \eta$.

On the other hand,
\begin{eqnarray*}
  \CL_Q \alpha_{\eta} &=& \sum \eta_{i_1,...,i_n}^j x^{i_1}\cdots x^{i_{l-1}}(m_{j_1,...,j_r}^{i_l} x^{j_1}\cdots x^{j_r}) x^{i_{l+1}} \cdots x^{i_n} dx^j \\
  & & +\sum \eta_{i_1,...,i_n}^j x^{i_1} \cdots x^{i_n} d(m_{j_1,...,j_r}^j x^{j_1} \cdots x^{j_r}).
\end{eqnarray*}
By comparing each coefficients, we obtain $\alpha_{b^* \eta} = \CL_Q \alpha_{\eta}$.

For the cyclic case, the Connes' complex $C^\lambda_\bullet(A)=C_\bullet(A,A)/1-t$ defines
the cyclic homology and similar arguments as above can be used to prove the desired identifications,
which we leave for the readers as an exercise.

Later, we will introduce an operation $\WT{\;}$, and then show that $\widetilde{b^* \eta}$ corresponds to $d\mathcal{L}_Q \eta$.
\end{proof}

\section{Kontsevich-Soibelman's theorem}\label{sec:KS}
We give a brief sketch of the proof of the Kontsevich-Soibelman's theorem for reader's convenience,
and refer reader's to \cite{KS} for more details.
\begin{proof}
Consider a symplectic form $\omega$, satisfying  $d_{cycl} \omega=0, \CL_Q(\omega)=0$, which is a
cycle of the complex $(\Omega_{cyc}^{2,cl}(X),\CL_Q)$, where $cl$ means $d_{cycl}$-closed elements

By the Poincar\'e lemma,
there exists an element $\alpha \in \Omega^1_{cyc}(X)/d_{cyc} \Omega^0_{cyc}(X)$ such that $d_{cycl} \alpha = \omega$.
This provides an isomorphism of complexes:
$$d_{cycl}:(\frac{\Omega^1_{cyc}(X)}{d_{cycl} \Omega^0_{cyc}(X)},\CL_Q) \to (\Omega_{cyc}^{2,cl}(X),\CL_Q).$$
We remark that as it is the isomorphism, there exists an inverse, but
we do not know any map from $\Omega_{cyc}^{2,cl}(X) \to  \Omega^1_{cyc}(X)$ which is a chain
map with respect to $\CL_Q$ which is a source of some complications. For example, the contracting homotopy in the proof of the Poincar\'e lemma does not commute with the differential $\CL_Q$. 

Kontsevich and Soibelman has proved that the following map via $a db \to [a,b]$
$$(\frac{\Omega^1_{cyc}(X)}{d_{cycl} \Omega^0_{cyc}(X)},\CL_Q) \to 
([\CO(X),\CO(X)]_{top},\CL_Q),$$
is a quasi-isomorphism.
From the definition $\Omega^0_{cyc}(X)=\CO(X)/[\CO(X),\CO(X)]_{top}$, 
we have a short exact sequence of $\CL_Q$-complexes, 
$$0 \to [\CO(X),\CO(X)]_{top} \to \CO(X)/\kk \to \Omega^0_{cyc}(X)/\kk \to 0.$$
Note that $(\CO(X)/\kk,\CL_Q)$ is acyclic (\cite{KS2} Prop. 8.4.1), hence 
$(\Omega_{cyc}^{2,cl}(X),\CL_Q)$ is quasi-isomorphic to $(\Omega^0_{cyc}(X)/\kk, \CL_Q)$ which is
the cyclic cohomology of $A$ (see Lemma \ref{lem:compare}).
\end{proof}

To show that the resulting cyclic structure really depends on the $\CL_Q$-cohomology class of $\omega$, they prove
\begin{lemma}[\cite{KS2} Lemma 11.2.6]\label{lemma:vQ}
Let $\omega_1 = \omega + \CL_Q(d\alpha)$. Then there exists a vector field $v$ such that
$v(x_0)=0$, $[v,Q]=0$, and $\CL_v(\omega) = \CL_Q(d\alpha)$.
\end{lemma}
\begin{proof}
As in the proof of Darboux lemma, we need to find a vector field $v$, satisfying the condition
$\CL_v \omega = \CL_Q(d\alpha)$. Let $\beta = \CL_Q(\alpha)$.
Then, $d\beta = \CL_Q(d\alpha)$.

Hence, the desired equation $\CL_v \omega = \CL_Q(d\alpha)$, is equivalent to
$$ d i_v \omega = d\beta = d \CL_Q(\alpha).$$
Hence, we solve $$i_v \omega = \beta,$$
which is possible by the non-degeneracy of $\omega$.

We also claim that any such solution $v$ automatically satisfies $[Q,v]=0$.
To see this, note that
$$\CL_Q \circ i_v\omega = \CL_Q \circ \CL_Q (\alpha) = \CL_{[Q,Q]} (\alpha) =0.$$
But the first term equals
$$\CL_Q \circ i_v\omega = i_v\CL_Q \omega + i_{[Q,v]} \omega =  i_{[Q,v]} \omega. $$
Hence, $i_{[Q,v]} \omega =0$ and this implies the claim
as $\omega$ is non-degenerate.
\end{proof}

The above lemma suggests that there exist an $\AI$-automorphism (preserving $\AI$-structure) which
transforms the symplectic form $\omega + \CL_Q(d\alpha)$ to $\omega$, thus proving that 
the cyclic structure depends only on the $\CL_Q$-cohomology class. But we found that the construction of such an automorphism is rather involved which occupies the whole section \ref{sec:details}.

\section{Explicit relations} \label{sec:negtoship}
In this section, we show that for a negative cyclic cocycle $\phi \in HC^{\bullet}_{-}(A,A^*)$ with a suitable non-degeneracy
condition, it gives rise to a strong homotopy inner product in a canonical way.
Denote the negative cyclic cycle $\phi$ as $\phi = \sum_{i\geq 0} \phi_i v^i$, where $v$ is a formal parameter of degree -2.
Here cocycle condition implies that we have  $b^*\phi_i = B^*\phi_{i+1}$ for each $i$.

First, we make the following observation.
\begin{prop}\label{prop:first}
    Let $\phi \in C^{\bullet}(A,A^*)$ be a negative cyclic cocycle.
    We define $$\widetilde{\phi_0}(\vec{a}, \underline{v}, \vec{b})(w):=\phi_0 (\vec{a},v,\vec{b})(w)-\phi_0 (\vec{b},w,\vec{a})(v).$$
    Then $\widetilde{\phi_0}$ is an $\AI$-bimodule map from $A$ to $A^*$,
    satisfying the skew-symmetry and closedness condition in the definition \ref{def:cyc2}.
\end{prop}
For convenience, we write both $\WT{\phi_0}=\WT{\phi}$ without distinction.
\begin{proof}
Recall that $\widetilde{\phi_0}$ is an $\AI$-bimodule map from $(C,m)$ to $(C^*,m^*)$ if
$$\WT{\phi_0} \circ \WH{m} = m^* \circ \widehat{\wtd{\phi_0}}.$$
We will show this in two steps.
\begin{lemma}\label{step1} We have
$$\WT{\phi_0} \circ \WH{m} - m^* \circ \widehat{\wtd{\phi_0}} = \WT{B^*\phi_1},$$
where $\WT{B^*\phi_1}$ is defined by
$$\WT{B^*\phi_1}(\vec{a},v,\vec{b})(w) = B^* \phi_1 (\vec{b},w,\vec{a})(v) - B^* \phi_1 (\vec{a},v,\vec{b})(w)$$
\end{lemma}

\begin{lemma}\label{Bvanish} We have
$$\WT{B^*\gamma}(\vec{a},v,\vec{b})(w) = B^* \gamma (\vec{b},w,\vec{a})(v) - B^* \gamma (\vec{a},v,\vec{b})(w)=0,$$ for any $\gamma \in C^{\bullet}(A,A^*)$,
 and for any $\vec{a},\vec{b},v,w$.
\end{lemma}
Combining the above two lemmas, we obtain the proposition.
The skew-symmetry and closedness condition is easy to check and its proof is omitted.
\end{proof}
\begin{proof}
We begin the proof of lemma \ref{step1}.
We first show that
\begin{equation}\label{eqtilde}
(\wtd{\phi_0} \circ \widehat{m}-m^* \circ \widehat{\wtd{\phi_0}})(\vec{a},\UL{v},\vec{b})(w)= b ^* \phi_0 (\vec{a},v,\vec{b})(w) - b ^* \phi_0 (\vec{b},w,\vec{a})(v).
\end{equation}
And this equals the following as $b^*\phi_0 = B^*\phi_1$ as it is negative cyclic cocycle.
$$B^* \phi_1 (\vec{b},w,\vec{a})(v) - B^* \phi_1 (\vec{a},v,\vec{b})(w).$$
We again omit the Koszul signs in the following formula and express the additional contributions of signs.
Let $\vec{a}:=(a_1,...,a_n)$ and $\vec{b}:=(b_1,...,b_m)$.
  \begin{eqnarray}\label{eqt:1}
    & & (\wtd{\phi_0} \circ \widehat{m})(\vec{a},\UL{v},\vec{b})(w) \\
    &=& \sum_{\stackrel{0 \leq i \leq n-1}{k \geq 1}} \phi_0(a_1,...,a_i,m_k(a_{i+1},...,a_{i+k}),a_{i+k+1},...,a_n,v,\vec{b})(w) \nonumber \\
    &+& \sum_{\stackrel{0 \leq i \leq n-1, 1\leq j \leq m}{k \geq 1}} \phi_0(a_1,...,a_i,m_k(a_{i+1},...,a_n,v,b_1,...,b_j),b_{j+1},...,b_m)(w) \nonumber \\
    &+& \sum_{\stackrel{0 \leq j \leq m-1}{k \geq 1}} \phi_0(\vec{a},v,b_1,...,b_j,m_k(b_{j+1},...,b_{j+k}),b_{j+k+1},...,b_m)(w) \nonumber \\
    &-& \sum_{\stackrel{0 \leq j \leq m}{k \geq 1}} \phi_0(b_1,...,b_j,m_k(b_{j+1},...,b_{j+k}),b_{j+k+1},...,b_m,w,\vec{a})(v) \nonumber \\
    &-& \sum_{\stackrel{0 \leq i \leq n-1}{k \geq 1}} \phi_0(\vec{b},w,a_1,...,a_i,m_k(a_{i+1},...,a_{i+k}),a_{i+k+1},...,a_n)(w) \nonumber \\
        &-& \sum_{\stackrel{0 \leq s \leq n-1,1 \leq j \leq m}{k \geq 1}} \phi_0(b_{j+1},...,b_m,w,a_1,...,a_s)(m_k(a_{s+1},...,a_n,v,b_1,...,b_j)) \nonumber
  \end{eqnarray}

  \begin{eqnarray}\label{eqt:12}
    & & (m^* \circ \widehat{\wtd{\phi_0}})(\vec{a},\UL{v},\vec{b})(w) \\
     &=&  \sum_{\stackrel{0 \leq j \leq m-1, 1 \leq i \leq n}{k \geq 1}} \phi_0(b_1,...,b_j,m_k(b_{j+1},...,b_m,w,a_1,...,b_i),a_{i+1},...,a_n)(v) \nonumber \\
     &-& \sum_{\stackrel{1 \leq i \leq n, 0 \leq j \leq m-1}{k \geq 1}} \phi_0(a_{i+1},...,a_n,v,b_1,...,b_j)(m_k(b_{j+1},...,b_m,w,a_1,...,a_i)) \nonumber
  \end{eqnarray}
  On the other hand,
  \begin{eqnarray}
    & & b^* \phi_0 (\vec{a},v,\vec{b})(w)-b^* \phi_0 (\vec{b},w,\vec{a})(v) \nonumber \\
    &=& (\ref{eqt:1})  -  (\ref{eqt:12}) \nonumber \\
    &+& \sum \phi_0(a_i,...,a_l)(m(a_{l+1},...,a_k,v,\vec{b},w,a_1,...,a_{i-1})) \label{eqt:2}\\
    &+& \sum \phi_0(b_j,...,b_p)(m(b_{p+1},...,b_n,w,\vec{a},v,b_1,...,b_{j-1})) \label{eqt:3}\\
    &-& \sum \phi_0(b_j,...,b_p)(m(b_{p+1},...,b_n,w,\vec{a},v,b_1,...,b_{j-1})) \label{eqt:4}\\
    &-& \sum \phi_0(a_i,...,a_l)(m(a_{l+1},...,a_k,v,\vec{b},w,a_1,...,a_{i-1})) \label{eqt:5}
  \end{eqnarray}

  Note that the terms (\ref{eqt:2})-(\ref{eqt:5}) cancel out by themselves.
  By combining the above results, the lemma \ref{step1} is obtained.
\end{proof}
\begin{proof}
Now  we prove lemma \ref{Bvanish}.
  \begin{eqnarray*}
    B^* \gamma (c_1,...,c_n)(c_{n+1}) &=& \sum_{\sigma \in \mathbb{Z} / n\mathbb{Z}} \gamma(c_{\sigma(1)},...,c_{\sigma(n+1)})(1)
  \end{eqnarray*}
  Hence, $B^* \gamma(c_1,...,c_n)(c_{n+1}) = B^* \gamma (c_{\sigma(1)},...,c_{\sigma(n)})(c_{\sigma(n+1)})$ for any $\sigma \in \mathbb{Z}/ n \mathbb{Z}$.
  In particular, $B^* \phi_1 (\vec{b},w,\vec{a})(v) - B^* \phi_1 (\vec{a},v,\vec{b})(w)=0$.
\end{proof}
\begin{remark}\label{rmkN}
In the case that we use the (dual of) Tsygan's bicomplex, instead of $(b^*,B^*)$-complex to define
the negative cyclic cohomology, the same proposition holds true:
this is because the equation \ref{eqtilde} still holds. If we have $b^*\phi_0 = N^*\phi_1'$ instead for
the symmetrization operator $N$, then the proof above shows that $\WT{N^*\phi_1}$ also should vanish
as in the case of $B$ using the same symmetry argument.
\end{remark}

Hence
if $\wtd{\phi}_{0,0}$ is nondegenerate on $H^{\bullet}(A)$, then $\phi$ indeed gives a strong homotopy inner product. We call such a $\phi \in C^{\bullet}_{-}(A,A^*)$ be \textit{homologically nondegenerate}(H.N. for short below).

\begin{lemma}\label{lemma:corr}
We have the following 1-1 correspondences.

\begin{tabular}{|c|c|}\hline
     $\AI$-algebra $A$ & Formal noncommutative manifold $X$ \\\hline \hline
    skew-sym. $\AI$-bimod. map $\psi:A \to A^*$ & $ \omega_\psi \in \Omega_{cyc}^2(X)$ with $L_Q \omega_\psi =0$ \\\hline
    $\eta \in C^\bullet(A,A^*)$ & $\alpha_\eta \in \Omega^1_{cyc}(X)$\\\hline
    $\widetilde{\eta}$ & $d \alpha_\eta$ \\\hline
    S.H.I.P. $\phi:A \to A^*$  & 
    H.N. $\omega_\phi \in \Omega_{cyc}^2(X), d \omega_\phi=0= \CL_Q \omega_\phi$ \\\hline
\end{tabular}
%
%
%
%There exist one to one correspondence between skew-symmetric $\AI$-bimodule maps $\psi$ and $L_Q$-closed two forms in $\Omega_{cyc}^2(X)$.
%For Hochschild cocycle $\eta \in C^\bullet(A,A^*)$ and for the corresponding $\alpha_\eta \in \Omega^1_{cycl}(X)$ in the Lemma \ref{lem:compare}, the $\AI$-bimodule map 
%$\widetilde{\eta}$ obtained from
%the previous proposition corresponds to $d \alpha_\eta$.
%And strong homotopy inner product $\phi:A \to A^*$  corresponds to 
% homologically nondegenerate 2-form  $\omega_\phi$ such that $d \omega_\phi=0$ and $\CL_Q \omega_\phi=0$
\end{lemma}
\begin{proof}
Given a collection of maps $\psi_{k,l}:A^{\otimes k} \otimes \UL{A} \otimes A^{\otimes l} \rightarrow A^*$, we assign a cyclic 2-form  $$\omega_{\psi}=\sum (\psi_{k,l}(e_{i_1},...,e_{i_k},\UL{e_j},e_{j_1},...,e_{j_l})(e_n))x^{i_1}\cdots x^{i_k} dx^j x^{j_1}\cdots x^{j_l} dx^n$$ for basis elements $e_*$ (as in \cite{C}). Skew-symmetry is needed as we cannot tell the order of $dx^j, dx^n$ in the expression for cyclic forms.

We omit the proof of the correspondence of $L_Q$-closedness and $\AI$-bimodule property. This can be
carried out similarly as in the proof of Prop \ref{prop:first} and Lemma \ref{lem:compare} and it is tedious but elementary computations.

We show that $\omega_{\widetilde{\eta}}=d\alpha_{\eta}$. Observe that
\begin{eqnarray*}
  \widetilde{\eta}(e_{i_1},...,e_{i_k},\UL{e_j},e_{j_1},...,e_{j_l})(e_n) &=& \eta(e_{i_1},...,e_k,e_j,e_{j_1},...,e_{j_l})(e_n) \\
  & & -\eta(e_{j_1},...,e_{j_l},e_n,e_{i_1},...,e_{i_k})(e_j) \\
  &=& \eta_{i_1,...,i_k,j,j_1,...,j_l}^n - \eta_{j_1,...,j_l,n,i_1,...,i_k}^j,
\end{eqnarray*}
so $$\omega_{\widetilde{\eta}}=\sum(\eta_{i_1,...,i_k,j,j_1,...,j_l}^n - \eta_{j_1,...,j_l,n,i_1,...,i_k}^j)x^{i_1}\cdots x^{i_k}dx^j x^{j_1} \cdots x^{j_l} dx^n$$ is the 2-form corresponding to $\widetilde{\eta}$.

By definition,
\begin{equation}\label{eq:dalpha}
d \alpha_{\eta}=\sum_l \sum \eta_{i_1,...,i_n}^j x^{i_1}\cdots dx^{i_l} \cdots x^{i_n} dx^j.\end{equation}
Note that in $\Omega_{cyc}$, we have (up to Koszul sign)
 $$x^{i_{l+1}} \cdots x^{i_n} dx^j x^{i_1} \cdots x^{i_{l-1}} dx^{i_l} = -x^{i_1} \cdots dx^{i_l} \cdots x^{i_n} dx^j,$$ and hence (\ref{eq:dalpha}) reduces to
$$d \alpha_{\eta} = \sum_l \sum(\eta_{i_1,...,i_n}^j - \eta_{i_{l+1},...,i_n,j,i_1,...,i_{l-1}}^{i_l})x^{i_1}\cdots dx^{i_l} \cdots x^{i_n} dx^j.$$
Then we have $\omega_{\eta}=d\alpha_{\eta}$ by rearranging indices above.

Suppose that we are given a strong homotopy inner product $\phi:A \to A^*$.
Consider the corresponding two form $\omega_{\phi}$ from the above. It is not hard to check that the
closedness condition is equivalent to $d_{cyc}\omega_{\phi} = 0$.
Hence, as we proved that $\CL_Q$-closedness of $\omega_\phi$ is equivalent to $\phi$ being  $\AI$-bimodule map,
so we obtain the last claim.
\end{proof}

\section{Construction of an automorphism}\label{sec:details}
In this section, we prove that two strong homotopy inner products obtained two negative cyclic cocycles in the
same homology class are indeed equivalent to each other in the sense of \ref{equivst} (see also the comments at the end of the section \ref{sec:KS}).

First, we construct $\AI$-automorphisms from certain kinds of vector fields.
\begin{lemma}\label{lem:homfromv}
A formal vector field $v$ which satisfies $[Q,v]=0$ provides
an $\AI$-automorphism. Here $v$ is assumed to have length $\geq 2$.
(i.e. any non-trivial component of $v$ which is given by $f(x)\PD{i}$ satisfies
$order(f(x)) \geq 2$).
\end{lemma}
\begin{proof}
A formal vector field $v$ (as a derivation) corresponds to a coderivation, which we also call $v$, of tensor coalgebra $TV[1]$. Such $v$
is represented by a family of maps $v_k:A^{\otimes k} \to A$, and denote by $\WH{v}$ the coderivation
$$\WH{v}:TV[1] \to TV[1], \WH{v} = \sum_k \WH{v_k}.$$
where $\WH{v_k}$ is defined as in the definition of $\AI$-operation $\WH{m_k}$.
Corresponding to the condition $[Q,v]=0$ is the identity
\begin{equation}\label{dffd}
\WH{d} \circ \WH{v} = \WH{v} \circ \WH{d}.
\end{equation}

We define its exponential $e^{\WH{v}}$ as
\begin{equation}
e^{\WH{v}} = 1 + \WH{v} + \frac{1}{2!}\WH{v} \circ \WH{v} +  \frac{1}{3!}\WH{v} \circ \WH{v} \circ \WH{v} + \cdots = \sum_{k=0}^{\infty}
\frac{1}{k!}(\WH{v})^k
\end{equation}
One can check that the infinite sum makes sense due to the assumption on $v$.
Let $\pi:TV[1] \to V[1]$ be the natural projection to its component of tensor length one.
Then, we define
\begin{equation}\label{aihomocomp}
f := \pi \circ e^{\WH{v}} : TV[1] \to V[1].
\end{equation}

It is easy to check that one may write
$$f = id \circ \pi + v (\sum \frac{1}{k!} (\WH{v})^{k-1}) .$$
In fact, by the assumption on $v$, $f_1:V[1] \to V[1]$ is given by identity.
 
For example, we have
\begin{eqnarray*}
  & & e^{\WH{v}}(x_1 \otimes x_2 \otimes x_3) \\
  &=& x_1 \otimes x_2 \otimes x_3 + v(x_1 \otimes x_2)\otimes x_3 + x_1 \otimes v(x_2 \otimes x_3) \\
  & & + v(x_1 \otimes x_2 \otimes x_3) + \frac{v(v(x_1 \otimes x_2) \otimes x_3) + v(x_1 \otimes v(x_2 \otimes x_3))}{2} \\
  &=& f_1(x_1) \otimes f_1(x_2) \otimes f_1(x_3) + f_2(x_1 \otimes x_2) \otimes f_1(x_3)+f_1(x_1) \otimes f_2(x_2\otimes x_3) \\
  & & + f_3(x_1 \otimes x_2 \otimes x_3) \\
  &=& \WH{f}(x_1 \otimes x_2 \otimes x_3).
\end{eqnarray*}

In general, we have $\WH{f} = e^{\WH{v}}$, which we prove in the following lemma.
Now, the proof of the Lemma \ref{lem:homfromv} follows from the following lemma.
\end{proof}
\begin{lemma}\label{homfromv2}
$f$ defines an $\AI$-automorphism. More precisely,
we have $$\WH{f} = e^{\WH{v}}, \;\; \WH{d}\WH{f} = \WH{f}\WH{d}.$$
\end{lemma}
\begin{proof}
We first show that $e^{\WH{v}}:TV \to TV$ satisfies the following identity
\begin{equation}\label{cohoeq}
(e^{\WH{v}} \otimes e^{\WH{v}} ) \circ \Delta = \Delta  \circ e^{\WH{v}}
\end{equation}
This would imply that $e^{\WH{v}}$ is a cohomomorphism, and it is well-known that such a cohomomorphism
is completely determined by its projection (\ref{aihomocomp}) (see for example \cite{T}) and
satisfies the identity $\WH{f} = e^{\WH{v}}$. 

To prove the identity, we apply (\ref{cohoeq}) to an expression $x_1\otimes \cdots \otimes x_k$. 
The left hand side of (\ref{cohoeq}) becomes
$$(e^{\WH{v}} \otimes e^{\WH{v}} ) \circ \Delta (x_1\otimes \cdots \otimes x_k)
= \sum_{i=1}^k \big(e^{\WH{v}}(x_1\otimes \cdots \otimes x_i) \otimes e^{\WH{v}}(x_{i+1}\otimes \cdots \otimes x_k) \big).$$
The right hand side becomes
$$ \Delta \circ e^{\WH{v}} (x_1\otimes \cdots \otimes x_k)
= \Delta \big( \sum_{j=0}^{\infty} \frac{1}{j!}(\underbrace{\WH{v} \circ \cdots \circ \WH{v}}_{j}(x_1\otimes \cdots \otimes x_k)) \big)$$
$$= \sum_{j=0}^{\infty}  \frac{1}{j!} \sum_{\stackrel{(j_1,j_2) shuffle}{j_1+j_2=j}} 
\big(\underbrace{\WH{v} \circ \cdots \WH{v}}_{j_1}\big) \otimes \big( 
   \underbrace{\WH{v} \circ \cdots \WH{v}}_{j_2} \big) \circ \Delta (x_1\otimes \cdots \otimes x_k).$$
The equality here is obtained by noting that $\Delta$ divides the tensor product into two parts.
Recall that the number of such shuffles are $\frac{j!}{j_1!j_2!}$ and hence the above expression becomes
$$= \sum_{j=0}^{\infty}  \sum_{j_1+j_2=j} \big(\frac{1}{j_1!} (\WH{v})^{j_1} \otimes 
 \frac{1}{j_2!} (\WH{v})^{j_2}\big) \circ \Delta (x_1\otimes \cdots \otimes x_k).$$
 This proves the claim.

From this, we have
$$\WH{d}\WH{f} = \WH{d}\circ e^{\WH{v}} = e^{\WH{v}}\circ \WH{d} =  \WH{f}\WH{d}.$$
by the identity (\ref{dffd}) above.
\end{proof}
\begin{remark}
The automorphism just defined is {\em not} the automorphism  to
transform $\omega + \CL_Q(d\alpha)$ to $\omega$ that is suggested in the Lemma \ref{lemma:vQ}.
In fact it is a first order approximation of the correct automorphism, and in the next proposition,
we show how to find the actual automorphism which transforms $\omega + \CL_Q(d\alpha)$ to $\omega$.
\end{remark}

In the section \ref{sec:negtoship}, we assigned a strong homotopy inner product to a negative cyclic cocycle. Now we prove that the assignment is also well-defined on the cohomology level up to equivalence of strong homotopy inner products.

\begin{prop}\label{prop:welldefined}
Let $A$ be a weakly unital compact $\AI$-algebra.
  If two negative cyclic cocycles $\phi$ and $\phi '$ give the same cohomology class, then $\widetilde{\phi}$ and $\widetilde{\phi '}$ are equivalent as strong homotopy inner products.
\end{prop}
\begin{proof}
First, we pull-back all the related notions to the minimal model $H^\bullet(A,m_1)$, which is unital and
finite dimensional. By using the decomposition theorem of an $\AI$-algebra, suppose we have
 $A=H\oplus A_{lc}$, where $H$ is the minimal part and $A_{lc}$ is the linear contractible part of $A$.
Let $i:H\to A$ be the inclusion, which is also an $\AI$-quasi-isomorphism.

First,in the unital case, as the two cycles $\phi, \phi '$ are cohomologous, we may write
$\phi^{'} = \phi + (b^* + vB^*) \psi$. Hence may write for some $\eta$, $\gamma \in C^{\bullet}_{-}(A,A^*)$ 
$$\phi^{'}_0 = \phi_0 + b^* \eta + B^* \gamma.$$
Hence, the induced $\AI$-bimodule maps from the Prop. \ref{prop:first} satisfy
$\WT{\phi^{'}}=
 \widetilde{\phi}+\wtd{b^* \eta}+\wtd{B^* \gamma}$, but by lemma \ref{Bvanish}, we have  $\wtd{B^* \gamma}\equiv 0$, so $\WT{\phi^{'}} = \widetilde{\phi}+\wtd{b^* \eta}$.
In the weakly unital case, one can proceed similarly using Tsygan's bicomplex using the remark \ref{rmkN}.

Now, using the $\AI$-quasimorphism $i:H \to A$, we pull back $\WT{\phi}$ and $\WT{\phi^{'}}$ to $H$ by
\begin{equation}\label{defdiagram5}
\xymatrix{A \ar[d]_{\WT{\phi}} & H \ar[l]_{\tilde{i}} \ar[d]_{i^* \WT{\phi}} \\ A^* \ar[r]^{{\tilde{i}}^*} & H^* }
\end{equation}
to obtain $i^* \WT{\phi}$ and $i^* \WT{\phi^{'}}$. From the definition of the equivalence of strong homotopy 
inner products, it is enough to prove the equivalence between 
$i^* \WT{\phi}$  and $i^* \WT{\phi^{'}}$.

Using $i$, we can also pull-back the Hochschild cohomology classes by
$i^*:C^\bullet(A,A^*) \rightarrow C^*(H,H^*)$.
We claim that 
$$i^* \widetilde{b^*_A \eta} = \widetilde{i^* b^*_A \eta} = \widetilde{b^*_H i^* \eta}$$
Here, the first $i^*$ was used to pullback an infinity inner product, while the other $i^*$'s are for Hochschild cochains. The first equality is almost trivial, and the second one is given by following:
\begin{eqnarray*}
  & & i^* b^*_A \eta(a_1,...,a_k)(a_{k+1}) \\
  &=& \sum b^*_A \eta(i(a_1),...,i(a_k))(i(a_{k+1})) \\
  &=& \sum \eta(i(a_1),...,i(a_l),m_A(i(a_{l+1}),...,i(a_p)),i(a_{p+1}),...,i(a_k))(i(a_{k+1})) \\
  & & +\sum \eta(i(a_j),...,i(a_p))(m_A(i(a_{p+1}),...,i(a_{k+1}),i(a_1),...,i(a_{j-1}))) \\
  &=& \sum \eta(i(a_1),...,i(a_l),i(m_H(a_{l+1},...,a_p)),i(a_{p+1}),...,i(a_k))(i(a_{k+1})) \\
  & & +\sum \eta(i(a_j),...,i(a_p))(i(m_H(a_{p+1}),...,i(a_{j-1}))) \\
  &=& b^*_H i^* \eta(a_1,...,a_k)(a_{k+1}).
\end{eqnarray*}
Observe that in the third equality we used the fact $i \circ \widehat{m_H}=m_A \circ \widehat{i}$, i.e. $i$ is an $\AI$-homomorphism.

By using the results of \cite{C}, in fact, we can pull them back further similarly via the diagram
\begin{equation}\label{defdiagram6}
\xymatrix{H \ar[d]_{cyc} \ar[r]_{\tilde{g}} & H  \ar[d]_{ \WT{\phi}} \\ H^*  & \ar[l]^{{\tilde{g}}^*} H^* }
\end{equation}
to assume that the strong homotopy inner product $\WT{\phi}$ is in fact cyclic inner product.

Therefore, it is enough to prove the proposition for the minimal model $H$ with the cyclic inner product
$\WT{\phi}$ and it suffices to find an $\AI$-automorphism $f$ with the following commutative diagram:
\begin{equation}\label{defdiagram8}
\xymatrix{H \ar[d]_{\widetilde{\phi}} \ar@{.>}[r]^{f} & H \ar[d]^{\widetilde{\phi}+\wtd{b^* \eta}} \\
            H^*  & H^* \ar@{.>}[l]^{f^*}}
            \end{equation}
            
It is very hard to get such an automorphism $f$ at once, so we need to construct it recursively. The construction becomes more natural if we use the dual notion of all above, namely formal noncommutative calculus. In the dual context, $H$ corresponds to a formal noncommutative affine manifold $X$, and $\AI$-automorphism $f$ corresponds to the coordinate change of $X$ preserving $Q$ which is a vector field corresponding to the $\AI$-structure of $H$ as before.

Let $\omega=\sum \omega_{ij} dx^i dx^j$ be a closed cyclic 2-form on $X$ which corresponds to the cyclic inner product $\widetilde{\phi}$. We denote 
$$d\mathcal{L}_Q \eta=\sum a_{ij}dx^j dx^i + \sum_{|I\cup J| \geq 1}a_{ij,IJ} x^I dx^i x^J dx^j.$$
We claim that the coefficients $a_{ij}=0$ for all $i,j$:
By minimality of $H$, $Q = 0 + O(x^2)$, i.e. the constant and the linear part of $Q$ is zero, and 
this implies the claim.

So by  the nondegeneracy of $\omega$, we can construct a vector field $v=\sum v_i(x)\PD{i}$ such that $\CL_v \omega = -d\CL_Q \eta$ 
or $i_v\omega = - \CL_Q \eta$.
 Observe that $v = 0 + O(x^2)$ and
$\omega$ and $\omega + d\CL_Q \eta$ have the same constant part, or 
$$\omega+ d\CL_Q \equiv \omega + O(x^2).$$
Using $v$, we constructed the automorphism $f$ in the previous lemma. To check how much $f$ has transformed
$\omega+ d\CL_Q$, we proceed as follows using non-commutative calculus.

First, we denote $$e^{\CL_v} := Id + \CL_v+\frac{(\CL_v)^2}{2!}+\frac{(\CL_v)^3}{3!}+\cdots.$$
\begin{lemma}
Under change of coordinates $x^i \mapsto e^{\CL_v}x^i$,
any differential form $\beta$ transforms as 
$$\beta \mapsto e^{\CL_v} \beta.$$

In fact the coordinate change here corresponds to an $\AI$-isomorphism of the lemma \ref{lem:homfromv}
in the sense of the (\ref{changeco}).
\end{lemma}
\begin{proof}
This is easily seen as follows. Since this is trivial for coordinate functions and $e^{\CL_v}$ commutes with $d$, it suffices to show that $e^{\CL_v}(\alpha \cdot \beta)=e^{\CL_v}\alpha \cdot e^{\CL_v}\beta$ for any two differential forms $\alpha$ and $\beta$.
\begin{eqnarray*}
  e^{\CL_v}\alpha \cdot e^{\CL_v}\beta &=& \sum_{k \geq 0} \frac{(\CL_v)^k}{k!}\alpha \cdot \sum_{l \geq 0}\frac{(\CL_v)^l}{l!}\beta \\
  &=& \sum_{k,l\geq 0} \frac{(\CL_v)^k}{k!}\alpha \cdot \frac{(\CL_v)^l}{l!}\beta
\end{eqnarray*}
and
\begin{eqnarray}
  e^{\CL_v}(\alpha \cdot \beta) &=& \sum_{k\geq 0} \frac{(\CL_v)^k}{k!}\alpha \cdot \beta \nonumber \\
  &=& \sum_{k \geq l \geq 0} \frac{1}{k!}(\CL_v)^{k-l} \alpha \cdot (\CL_v)^l \beta \cdot \frac{k!}{(k-l)!l!}. \label{eq:elv}
\end{eqnarray}
In (\ref{eq:elv}), we used that $\CL_v$ is a derivation, and $\displaystyle \frac{k!}{(k-l)!l!}$ means the number of $(k-l,l)$-shuffles. Hence we get the desired result.

To prove the second claim, let $f:H \rightarrow H$ be an $\AI$-automorphism such that
$$f(e_{i_1},...,e_{i_k})=\sum_j f_{i_1,...,i_k}^j e_j.$$
Then the coordinate change associated to $f$ is given by
$$x^j \mapsto \sum f_{i_1,...,i_k}^j x^{i_1} \cdots x^{i_k},$$
where $\{x^j\}$ is the dual coordinate of $\{e_j\}$.

Now let $f$ be given as in lemma 7.1, i.e. $f = id \circ \pi + v (\sum \frac{1}{k!} (\WH{v})^{k-1}) .$ As usual, let $$v(e_{i_1},...,e_{i_k})=\sum_j v_{i_1,...,i_k}^j e_j,$$
and let $v^j(e_{i_1},...,e_{i_k}):=v_{i_1,...,i_k}^j e_j.$ Then
\begin{equation*}\label{eq:fk}
f_k(e_{i_1},...,e_{i_k})=\sum_{1 \leq l \leq k-1}\frac{1}{l!}v \circ \widehat{v}^{l-1}(e_{i_1},...,e_{i_k}).
\end{equation*}
As above, let \begin{equation}\label{eq:fj} f^j_k(e_{i_1},...,e_{i_k}):=\sum_{1 \leq l \leq k-1}\frac{1}{l!}v^j \circ \widehat{v}^{l-1}(e_{i_1},...,e_{i_k}).\end{equation}

Finally, compare the coefficient of the $l$-th summand of (\ref{eq:fj}) and that of $\displaystyle \frac{(\CL_v)^l}{l!}x^j$, then we will easily get the result.
\end{proof}

Hence, this coordinate change gives us
$$\omega^{(1)}:=\omega+d\CL_Q \eta \mapsto \omega^{(2)}:=e^{\CL_v}\omega + e^{\CL_v}(d\CL_Q \eta),$$
\begin{eqnarray*}
  e^{\CL_v}\omega + e^{\CL_v}d\CL_Q \eta &=& (\omega + \CL_v \omega + \sum_{k\geq 2}\frac{1}{k!}(\CL_v)^k \omega)+ (d\CL_Q \eta + \sum_{k \geq 1}\frac{1}{k!}(\CL_v)^k d\CL_Q \eta) \\
  &=&\omega+\sum_{k\geq 2} \frac{1}{k!} (\CL_v)^{k-1} \CL_v \omega + d\CL_Q \sum_{k \geq 1}\frac{1}{k!}(\CL_v)^k \eta \\
  &=&\omega+\sum_{k\geq 2} \frac{1}{k!} (\CL_v)^{k-1} (-d\CL_Q \eta) + d\CL_Q \sum_{k \geq 1}\frac{1}{k!}(\CL_v)^k \eta \\
  &=&\omega+d\CL_Q \sum_{k \geq 2} (-\frac{1}{k!}(\CL_v)^{k-1} \eta)+d\CL_Q \sum_{k\geq 1}\frac{1}{k!}(\CL_v)^k \eta \\
  &=&\omega+d\CL_Q \sum_{k\geq 1} a_k (\CL_v)^k \eta \\
   &=&\omega+ \sum_{k\geq 1} a_k (\CL_v)^k (d\CL_Q \eta )
\end{eqnarray*}
for some numbers $a_k \in k$.
We emphasize that for the second and the fourth identities, we used lemma \ref{lemma:vQ} so that $[\CL_Q,\CL_v]=\CL_{[Q,v]}=0$.

Note that the term $d\CL_Q \eta$ changed into $\sum_{k\geq 1} a_k (\CL_v)^k (d\CL_Q \eta )$.
The operation $\CL_v = d \circ i_v + i_v \circ d$ increase the number of formal variable $x^i$'s in the
expression at least by one.

Hence, we have
$$\omega^{(2)} \equiv \omega + O(x^3).$$
By repeating the same procedure, we can  transforms $\omega + d\CL_Q \eta$ into $\omega$ via
countably many procedures.
We remark that the infinite composition of such automorphism is well-defined as the automorphism
at the step $(k)$ will fix the tensor product of length up to $(k)$.
This proves the proposition.
\end{proof}
Summarizing this provides proof of the Theorem \ref{thm}.

\section{A connection to the Kontsevich-Soibelman's result}\label{sec:compare}
In \cite{KS}, Kontsevich-Soibelman has provided the formula for the cyclic inner product on 
the minimal model using the trace, and we show that it agrees with our formula.

Namely, we have two ways to get cyclic inner products on $H^{\bullet}(A)$ from given a homologically nondegenerate negative cyclic cocycle $\phi$. Namely, for $a,b \in H^{\bullet}(A)$, we may consider $\wtd{\phi}(a)(b)=\phi(a)(b)-\phi(b)(a)$ as in proposition \ref{prop:first}, or consider $\omega (a)(b)=Tr_{c[\phi]}(m_2(a,b))=Tr_{[B^* \phi_0]}(m_2(a,b))$ as in \cite{KS}. $Tr_{[\eta]}:A/[A,A] \to k$ for $[\eta]\in HC^{\bullet}$ is given by
 $$Tr_{[\eta]}(a)=\eta_0|_{A^*}(a)$$
  (Recall that $\eta_0 \in \displaystyle C^{\bullet}(A,A^*)=\bigoplus_{n \geq 0}\mathrm{Hom}(A^{\otimes n},A^*)=\bigoplus_{n \geq 1}\mathrm{Hom}(A^{\otimes n},k)$).
\begin{prop}
  Let $\phi$ be a negative cyclic cocycle of $A$ with whose zeroth column part is $\phi_0$.
  Then $Tr_{[B^* \phi_0]}(m_2(\cdot, \cdot)) = \wtd{\phi}(\cdot)(\cdot)$.
\end{prop}

\begin{proof}
  We identify cocycles in $\bigoplus_{n \geq 1}Hom(A^{\otimes n},k)$.
  For $a,b\in H^{\bullet}(A)$,
  \begin{eqnarray*}
    & & Tr_{[B^* \phi_0]}(m_2(a,b)) \\
    &=& B^* \phi_0 (m_2(a,b)) \\
    &=& \phi_0(1,m_2(a,b)).
  \end{eqnarray*}
  A priori, we have $b^*\phi(1,a,b)=B^*\psi(1,a,b)$ for some hochschild cochain $\psi$ because $\phi$ is a cocycle. Clearly the right hand side is zero. On the other hand,
  \begin{eqnarray*}
    & & b^* \phi_0 (1,a,b) \\
    &=& \phi_0(\widehat{b}(1,a,b)) \\
    &=& \phi_0(m_2(1,a),b)+(-1)^{1 \cdot 1}\phi_0(1,m_2(a,b))+(-1)^{|b|'(|a|'+1)}\phi_0(m_2(b,1),a) \\
    &=& \phi_0(a,b)-\phi_0(1,m_2(a,b))+(-1)^{|a|'|b|'+|b|'+|b|}\phi_0(b,a). \\
  \end{eqnarray*}
  Hence $Tr_{[B^* \phi]}(m_2(a,b))=\phi_0(1,m_2(a,b))=\phi_0(a,b)-(-1)^{|a|'|b|'}\phi_0(b,a)$ as we desired.
\end{proof}

We remark that a minimal model of an $\AI$-algebra also has many automorphisms which do not preserve the $\AI$-structure and the cyclic structure, hence given an arbitrary minimal model, one can {\em not} assume that the trace as above provides the cyclic inner product of the given minimal model. Rather, \cite{KS} proves the existence of one minimal model which is cyclic with respect to the trace. Our formula provides a diagram to connect cyclic structure, (negative) cyclic cohomology class and the related $\AI$-structures.

We also remark that the homological non-degeneracy of cyclic cohomology class $\phi$ does {\it not} imply that $\phi$ is a non-trivial
cohomology class. For example, there exists an $\AI$-algebra with  trivial $m_1$-homology, but equipped with cyclic inner product.
In such a case, cyclic cohomology can be shown to be trivial using the spectral sequence arguments with the length filtration.

\section{Gapped filtered cases}\label{sec:filter}
Gapped filtered $\AI$-algebras are introduced by Fukaya, Oh, Ohta and Ono in their construction of gapped filtered $\AI$-algebra of Lagrangian submanifold. For the gapped filtered $\AI$-algebras, many of the results in this paper remain true as it will be explained. But there exists some subtlety in filtered notions, as sometimes non-negativity of the energy from the filtration is needed.  For example, the Darboux theorem 
in the general form does not hold true, but only for non-negative symplectic forms.

\subsection{Filtered $\AI$-algebras}
We recall the notion of gapped filtered $\AI$-algebra, and we refer readers to \cite{FOOO} for full details.
To consider $\AI$-algebras arising from the study of Lagrangian submanifolds or in general pseudo-holomorphic curves,
one considers filtered $\AI$-algebras over Novikov rings, where the filtration is given by the energy of pseudo-holomorphic curves.
Here Novikov rings are, for a ring $R$
(here $T$ and $e$ are formal parameters)
$$\NOV = \{ \sum_{i=0}^\infty a_i T^{\lambda_i}e^{q_i} | \; a_i \in R,\;\lambda_i \in \RR,\; q_i \in \ZZ, \; \lim_{i \to \infty} \lambda_i = \infty \}$$
$$\NOVO = \{ \sum_i a_i T^{\lambda_i}e^{q_i} \in \NOV | \lambda_i \geq 0 \}.$$

When we take dualizations, it is convenient to work with Novikov fields. The above rings $\NOV,\NOVO$ are not fields but one can forget the formal parameter $e$ (and work with $\ZZ/2$ grading only) and work with the following Novikov fields
\begin{equation}\label{def:lambda}
\Lambda =\{ \sum_{i=0}^\infty a_i T^{\lambda_i} | \; a_i \in \kk,\;\lambda_i \in \RR,\; \; \lim_{i \to \infty} \lambda_i = \infty \},\;\;
\Lambda_0 =\{ \sum_{i=0}^\infty a_i T^{\lambda_i} \in \Lambda| \lambda_i \geq 0 \}.
\end{equation}
Here, we consider a field $\kk$ containing $\QQ$, and there also exist another choice $\NOV^{(e)}$ in \cite{C}.
We remark that in most of the construction of \cite{FOOO}, they work with $\NOVO$ and only when one needs to work with $\NOV$, they take tensor product  $\otimes \NOV$ to work with $\NOV$ coefficients. We take a similar approach for $\Lambda$ and $\Lambda_0$.

The gapped condition is defined as follows.
The monoid $G \subset \RR_{\geq 0} \times 2 \ZZ$ is assumed to satisfy the following conditions
\begin{enumerate}
\item The projection $\pi_1(G) \subset \RR_{\geq 0}$ is discrete.
\item $G \cap (\{0\} \times 2\ZZ) = \{(0,0)\}$
\item $G \cap (\{\lambda \} \times 2\ZZ)$ is a finite set for any $\lambda$.
\end{enumerate}

Consider a free graded $\NOVO$ module $C$, and let $\OL{C}$ be an $\kk$-vector space
such that $C = \OL{C} \otimes_{\kk} \NOVO$. Then $(C,m_{\geq 0})$ is said to be $G$-gapped if there exists
homomorphisms $m_{k,\beta}:(\overline{C}[1])^{\otimes k} \to \overline{C}[1]$ for $k=0,1,\cdots,$ $\beta=(\lambda(\beta),\mu(\beta)) \in G$
such that
$$m_k = \sum_{\beta \in G} T^{\lambda(\beta)}e^{\mu(\beta)/2} m_{k,\beta}.$$
One defines filtered gapped $\AI$-algebras as in the definition \ref{def:ai}, by considering
the same equation \ref{aiformula} for $k=0,1,\cdots$.

Recall that these $m_k$ operations may be considered as coderivations by defining
\begin{equation}\label{eq:hatdfil}
\WH{m}_k(x_1 \otimes \cdots \otimes x_n) = \sum_{i=1}^{n-k+1}
(-1)^{|x_1|' + \cdots + |x_{i-1}|'} x_1 \otimes \cdots \otimes m_k(x_i,
\cdots, x_{i+k-1}) \otimes \cdots \otimes x_n
\end{equation}
for $k \leq n$ and $\WH{m}_k(x_1 \otimes \cdots \otimes x_n) =0$ for $k >n$.
If we set $\WH{d} = \sum_{k=0}^\infty \WH{m}_k$, the $\AI$-equations are equivalent to the equality
$\WH{d} \circ \WH{d} =0$.

We recall cyclic $\AI$-algebras in the gapped filtered case.

\begin{definition}
A filtered gapped $\AI$-algebra $(C,\{m_*\})$ is said to have a {\it cyclic symmetric} inner product if
there exists a skew-symmetric non-degenerate, bilinear map $$<,> : \overline{C}[1] \otimes \overline{C}[1] \to \kk,$$
which is extended linearly over $C$,
such that for all integer $k \geq 0$, $\beta \in G$,
\begin{equation}\label{cyeqnfil}
    <m_{k,\beta}(x_1,\cdots,x_k),x_{k+1}> = (-1)^{K}<m_{k,\beta}(x_2,\cdots,x_{k+1}),x_{1}>.
\end{equation}
 where $K = |x_1|'(|x_2|' + \cdots +|x_{k+1}|')$.
 For short, we will call such an algebra, cyclic (filtered) $\AI$-algebra.
\end{definition}

\subsection{Weakly filtered $\AI$-bimodule homomorphisms}
The usual notions of filtered $\AI$-homomorphisms, and filtered bimodule maps require the maps to preserve filtrations.
But the map obtained via differential forms in the Lemma \ref{lemma:corr} do not always preserve the filtration, but
provides so called, weakly filtered $\AI$-bimodule homomorphisms in \cite{FOOO}.

First we recall the notion of filtered $\AI$-homomorphism between two filtered $\AI$-algebras.
The family of maps of degree 0
$$f_k : B_k(C_1) \to C_2[1] \;\;\textrm{for} \; k =0,1,\cdots $$
induce the coalgebra map
$\HH{f}:\HH{B}C_1 \to \HH{B}C_2$, which for $x_1 \otimes  \cdots \otimes x_k \in B_k C_1$ is defined
by the formula
$$\HH{f}(x_1\otimes \cdots \otimes x_k) = \sum_{0 \leq k_1 \leq \cdots \leq k_n \leq k}
f_{k_1}(x_1,\cdots,x_{k_1})\otimes \cdots \otimes f_{k-k_n }(x_{k_n+1},\cdots,x_{k}).$$
We remark that the above can be an infinite sum due to the possible existence of $f_0(1)$.
In particular, $\HH{f}(1) = e^{f_0(1)}$. It is assumed that
\begin{equation}\label{fenergy}
\begin{cases}
f_k(F^\lambda B_k(C_1)) \subset F^\lambda C_2[1], \;\;\textrm{and}\\
f_0(1) \in F^{\lambda'}C_2[1] \,\;\; \textrm{for some}\; \lambda' >0. 
\end{cases}
\end{equation}
The map $\HH{f}$ is called a filtered $\AI$-homomorphism if  
$$\HH{d} \circ \HH{f} = \HH{f} \circ \HH{d}.$$

We recall the definition of weakly filtered $\AI$-bimodule homomorphisms from \cite{FOOO} in a
simple case of $A$-bimodules for an $\AI$-algebra $A=(C,\{m\})$.
Let $\WT{M}$ and $\WT{M}'$ be  filtered $(A,A)$ $\AI$-bimodules over $\NOV$.
A {\it weakly filtered $\AI$-bimodule homomorphism} $\WT{M} \to \WT{M}'$
is a family of $\NOV$-module homomorphisms
$$\phi_{k_1,k_0}:B_{k_1}(C) \HH{\otimes} \WT{M} \HH{\otimes} B_{k_0} (C) \to \WT{M}'$$
with the following properties:

\begin{enumerate}
\item There exists $c\geq 0$ independent of $k_0,k_1$ such that
$$\phi_{k_1,k_0} \big( F^{\lambda_1}B_{k_1}(C) \HH{\otimes} F^{\lambda}\WT{M} \HH{\otimes} F^{\lambda_0}B_{k_0} (C) \big) \subset F^{\lambda_1 +\lambda + \lambda_0 - c}\WT{M}'$$
\item  $\HH{\phi} \circ \HH{d} = \HH{d}' \circ \HH{\phi}$
\end{enumerate}

Weakly filtered homomorphisms arise when we study the
invariance property of the Floer cohomology $HF(L_0,L_1) \cong HF(L_0,\phi(L_1))$ where the constant $c$ is
related to the Hofer norm of the Hamiltonian isotopy $\phi$.

\subsection{Formal manifolds}
The bar complex $\WH{B}C$ in the filtered case is obtained by taking a completion with respect to energy.
Hence, the Hochschild complex $C_\bullet(A,A)$ is similarly defined but also has to be completed.
To consider dualization of the bar complex $\WH{B}C$, we consider only Novikov fields $\Lambda$, and also assume that $C$ is a finite dimensional vector space. And then, we can take topological dual spaces as in \cite{C}:
Let $V$ be a vector space over the field $\Lambda$ with finitely many generators $\{e_i\}_{i=1}^n$.
Consider $V$ as a topological vector space
by defining a fundamental system of neighborhoods of $V$ at $0$:
first define the filtrations $F^{>\lambda}V$ as
$$F^{>\lambda}V = \{ \sum_{j=1}^k a_j v_{i_j}| a_i \in \Lambda,\tau(a_i) > \lambda, \; \forall i\}.$$
Here $\tau$ is the valuation of $\Lambda$ which gives the minimal exponent of $T$.
We regard $F^{>\lambda}V$ for $\lambda=0,1,2,\cdots$ as fundamental system of neighborhoods at $0$.
The completion with respect to energy can be also considered as a completion using the Cauchy sequences in $V$ with the above topology.

Now, consider the following topological dual space
$$\CO(X) = Hom_{cont} (\WH{B}C, \Lambda).$$
Consider the dual basis $\{x_i\}_{i=1}^n$  considered as elements in
$\HH{V}^* = Hom_{cont}(\HH{V}, \Lambda).$
Then, the Lemma 9.1 of \cite{C1} may be translated as
\begin{lemma} We have
\begin{equation}\label{co}
\CO(X) = \Lambda <<x_1,\cdots,x_n>>,
\end{equation}
where the right hand side is the set of all formal power series of variables $x_1,\cdots,x_n$
whose coefficients in $\Lambda$ are bounded below.
\end{lemma}
In particular, $\CO(X)$ does not contain formal power series whose energy of the coefficients converging to
$-\infty$. Intuitively, the dual elements are allowed to have
infinite sums with bounded energy since the inputs for the evaluation already
has energy converging to infinity in its infinite sum.

One can also possibly use (\ref{co}) as a definition with several
different choices of coefficient rings $\NOV,\NOVO,\Lambda,\Lambda_0$. From now on,
we will work with $\Lambda$ but other coefficients can be used for the rest of the paper also with
little modification.

\subsection{Darboux theorem}
First, we define de Rham complex $\Omega_{cyc}(X)$, vector field $Q$ as before.
Note that the coefficients of the vector field $Q$ always have non-negative energy from the
definition of $\AI$-structure. Also note that $Q$ may have a component of constant vector field
which corresponds to the term $m_0$.

We show that the Darboux theorem in general does not hold in the filtered case, and one should restrict to
symplectic forms with non-negative energy.
Let $\omega \in \Omega_{cyc}^2(X)$ be a closed non-degenerate two form in the filtered setting as above.
Suppose the symplectic form can be written as $\omega = \omega_{ij}dx^idx^j + \omega'$
for $\omega' \in \Omega_{cyc}^2(X)$ such that each term of $\omega'$ has either positive energy ($T^\lambda$ for $\lambda>0$) or
positive length (with possibly negative energy).

\begin{theorem}[Darboux theorem]\label{darboux}
Consider the symplectic form $\omega = \omega_{ij}dx^idx^j + \omega'$ as above. If 
$\omega'$ does not contain a term with negative energy, then there exist filtered $\AI$-isomorphism $f$ which solves Darboux theorem. 
$$i.e. \;\;f^*\omega =  \omega_{ij}dx^idx^j.$$
But if $\omega'$ contains a term with negative energy with
positive length, then there does {\em not} exist any filtered $\AI$-isomorphism $f$ solving the Darboux theorem. 
\end{theorem}
\begin{proof}
For the first claim, we follow the proof of unfiltered case in the theorem 4.15 of \cite{Kaj}. 
In the gapped filtered case, the induction
should be run over the sum of two indices. As $\pi_1(G)$ is
discrete, we can find an increasing sequence $\lambda_j$ with $\lim \lambda_j  = \infty$ which covers
the image of $\pi_1(G) \subset \RR_{\geq 0}$.
We run the induction over the sum $k+j= N$, where $k$ is the power of $x^i$'s and 
 $j$ is for the energy level $\lambda_j$. 
 
Now, assume that $\omega$ satisfies the assumption, and $\omega$ is
transformed to the constant up to level $N$. Then, we consider the transformation of the form
$$ x^i \mapsto x^i + f^i,\;\;\; f^i = \sum_{j+k=N} T^{\lambda_j} x^{i_1}\cdots x^{i_k}.$$
By this transformation, $\omega$ is transformed as
$$\big( \omega_{ij}dx^idx^j + \omega_N + \cdots \big) \longmapsto
\big(\omega_{ij}dx^idx^j + \omega_N  + \omega_{ij}2d_{cycl}((f^i)dx^j)_c +\cdots \big).$$
But as $\omega_{ij}dx^idx^j + \omega_N + \cdots$ is $d_{cycl}$-closed, hence 
$\omega_N$ is $d_{cycl}$-closed and hence $d_{cycl}$-exact. So, by appropriate choice of $f^{i}$,
$\omega_N$ can be cancelled out as $\omega_{ij}$ is non-degenerate. Thus $\omega$ is transformed to be
constant up to $(N+1)$-level. Repeating this process completes the proof.

For the second statement, it is enough to show that a filtered isomorphism 
preserve the minimal negative exponent of the given symplectic form.
Note that as $f$ is an isomorphism, $f_1$ is an isomorphism. 
Then it is not hard to see as in the above that from the contribution of $f_1$, 
the change of coordinate by filtered $\AI$-map $f$ preserve the minimal negative exponent of the given symplectic form.
\end{proof}
We remark that there does not exist a notion of weakly filtered $\AI$-homomorphism.
Namely, a component $f_k$ of the filtered $\AI$-map $f$ cannot decrease the energy. 
If $f_k$ does decrease the energy, $\WH{f_k}$ for the bar complex would provide sequence of terms with the energy converging to $-\infty$,
but such elements do not exist in the bar complex $\WH{B}C$.
\subsection{Correspondences}
First, the definition of Hochschild (co)homology of filtered $\AI$-algebra can be given
in a similar way. But to consider its homological algebra, one has to be careful to deal with
$m_0$ terms, which we refer readers to \cite{C}. (For example, the standard contracting homotopy for the bar complex has to be modified.)
One define also the Hochschild cochain complex
$(C^\bullet(A,A^*),b^*)$ by taking the topological dual of $(C_\bullet(A,A),b)$.

As in the lemma \ref{lem:compare}, we have
\begin{lemma}\label{lem:compare2}
For a unital finite dimensional filtered gapped $\AI$-algebra $A$,
the complex $(\Omega^1_{cycl}(X)[1], \CL_Q)$ can be identified with Hochschild cochain complex $(C^\bullet(A,A^*),b^*)$,
and $(\Omega^0_{cycl}(X)/\Lambda, \CL_Q)$ can be identified with cyclic cochain complex $((C^\lambda(A))^*,b^*)$.
\end{lemma}
Also the Prop. \ref{prop:first} holds true in the gapped filtered case. The lemma \ref{lemma:corr}
has to be modified as follows. First, we denote by
$$ \Omega_{cyc,+}^2(X) \subset  \Omega_{cyc}^2(X),$$
the subset consisting of formal sums each term of which has non-negative energy coefficient.
\begin{lemma}\label{lemma:corr2}
We have the following 1-1 correspondences.
\begin{enumerate}
\item A filtered (resp. weakly filtered) skew-symmetric $\AI$-bimodule map $\psi:A \to A^*$ corresponds to
a two form $ \omega_\psi \in \Omega_{cyc,+}^2(X)$ (resp. $\in \Omega_{cyc}^2(X)$) with $L_Q \omega_\psi =0$.
\item The relation between $\widetilde{\eta}$  and $d \alpha_\eta$ is as before.
\item The strong homotopy inner product $\phi:A \to A^*$ corresponds to the homologically non-degenerate $\omega_\phi \in \Omega_{cyc,+}^2(X)$ with
$d \omega_\phi=0= \CL_Q \omega_\phi$.
\end{enumerate}
\end{lemma}

We remark that weakly filtered $\AI$-bimodule maps can be used to prove the following lemma, which is proved in \cite{C}.
\begin{lemma}
The homologically non-degenerate weakly filtered $\AI$-bimodule map $\phi:A \to A^*$
provides an isomorphism of Hochschild homology $H_\bullet(A,A)$ with $H_\bullet(A,A^*)$.
\end{lemma}

\subsection{The main theorem in the filtered case}
Now, we prove the main theorem for gapped filtered $\AI$-algebras.
First, we call a filtered $\AI$-algebra $(C,m)$ {\em compact} if the homology $H^\bullet(C,m_{1,0})$ is finite dimensional, and is called {\em canonical} if $m_{1,0} \equiv 0$. In \cite{FOOO}, the canonical model theorem (similar to minimal model theorem) is proved. The decomposition theorem in the filtered case is proved in \cite{CL}.

\begin{theorem}\label{thm2}
  For a weakly unital compact gapped filtered $\AI$-algebra $A$, a homologically nondegenerate negative cyclic cohomology class $[\phi]$, each term of which has non-negative energy, gives rise to an isomorphism class of strong homotopy inner products on $A$. In particular, from $[\phi]$, we construct a  strong homotopy inner product $ \WT{\phi_0}:A \to A^*$ explicitly using the Proposition \ref{prop:first}.
 
 In particular, we have a quasi-isomorphic cyclic gapped filtered $\AI$-algebra $B$ with $\psi:B \to B^*$ satisfying the commuting diagram \begin{equation}\label{defdiagram9}
\xymatrix{A \ar[d]_{\WT{\phi_0}} & B \ar[l]_{\WT{\iota}} \ar[d]^{cyc}_{\psi} \\ A^* \ar[r]^{\WT{\iota}^*} & B^* }
\end{equation}
\end{theorem}

\begin{proof}
The correspondence can be proved  using the lemma in the previous section. Hence it is enough to prove that
cohomologous  negative cyclic homology classes provide equivalent strong homotopy inner products.
As before, we pull-back all the related notions to the canonical model $H^\bullet(C,m_{1,0})$, which is unital and
finite dimensional.  

In the unital case, as the two cycles $\phi, \phi '$ are cohomologous, we may write
$\phi^{'} = \phi + (b^* + vB^*) \psi$. Hence may write for some $\eta$, $\gamma \in C^{\bullet}_{-}(A,A^*)$ 
$$\phi^{'}_0 = \phi_0 + b^* \eta + B^* \gamma.$$
Here we also assume that each term of $\eta$ and $\gamma$  also has non-negative energy.
So we have $\WT{\phi^{'}} = \widetilde{\phi}+\wtd{b^* \eta}$ as before.
To find a filtered $\AI$-automorphism $f$ satisfying the diagram \ref{defdiagram6},
we proceed as before but only modify the inductive argument using sum of order and energy.

In fact, we run the induction over the sum $k+2j= N$, where $k$ is the power of $x^i$'s and 
 $j$ is for the energy level $\lambda_j$ as in the proof of the theorem \ref{darboux}.
The reason that we use $2j$ instead of $j$ is as follows.

First, given a differential form 
$$c T^{\lambda_j} x_{i_{11}}\cdots x_{i_{1a_1}}dx_{j_1}x_{i_{21}}\cdots x_{i_{2a_2}}dx_{j_2}\cdots dx_{j_{m-1}}x_{i_{m1}}\cdots x_{i_{ma_m}}$$
we define its order to be $2j + a_1 + \cdots a_m$.
One can note that $d$ decrease the order by one, and $i_Q$ for the canonical model, increase the order by at least two.
Hence, the Lie derivative $\CL_Q=d \circ i_Q + i_Q \circ d$ increases the order by one.

Now, we will work with formal vector field $v$ such that $\CL_v \omega = -d\CL_Q \eta$ as before.
Note that such a $v$ can be chosen without a constant vector field term.
Then, the following can be proved analogously:
\begin{lemma}\label{lem:homfromv2}
In the filtered case, a formal vector field $v$ which satisfies $[Q,v]=0$ provides
an $\AI$-automorphism. Here $v$ is assumed to have order $\geq 2$ and no constant vector field term.
(i.e. any non-trivial component of $v$ which is given by $T^{\lambda_j} f_i(x)\PD{i}$ satisfies
$\big( order (f_i(x)) + 2j \geq 2 \big)$ and $f_i(x)$ is not constant.)
\end{lemma}

The rest of proof works as in the unfiltered case. In this case also,
the automorphism $f$ constructed above will change the symplectic form as
$$\omega + d\CL_Q \eta \longmapsto \omega + \sum_{k \geq 1} a_k (\CL_v)^k(d\CL_Q \eta)$$
But, note that $v$ has at least have order two. Hence,
$\CL_v = d \circ i_v + i_v \circ d$ increase the order at least by one.
Hence even in the gapped filtered case, the induction works as in the unfiltered case.
\end{proof}

\bibliographystyle{amsalpha}

\end{document}